\documentclass{article}
\usepackage[a4paper, total={6in, 8.5in}]{geometry}
\usepackage{graphicx}
\usepackage{amsmath,amsthm,amsfonts,amssymb,amscd}
\usepackage[section]{algorithm}
\usepackage{algorithmicx}
\usepackage{algpseudocode}
\usepackage{xcolor}
\usepackage{subfig}
\usepackage{multirow}
\usepackage{diagbox}
\newtheorem{theorem}{Theorem}
\newtheorem{lemma}{Lemma}

\title{Applying acceleration to Krylov subspace eigenvalue solvers}
\author{Michelle Baker, Sara Pollock}

\begin{document}
\maketitle

\begin{center}
    {\bf Abstract}
\end{center}

    In this paper, we apply acceleration to the inverse-free preconditioned Krylov subspace method introduced by Golub and Ye, which solves the symmetric generalized eigenvalue problem for the algebraically smallest eigenvalue. As the method is an improvement on steepest descent, we consider acceleration based on Nesterov accelerated steepest descent and Polyak's heavy-ball method. We extend acceleration to the block version of the Krylov subspace method and prove convergence for a more generalized choice of subspace. We present numerical results demonstrating the effect of fixed and safeguarded-adaptive choice of the momentum parameter, which show convergence in fewer outer iterations compared with LOBPCG with the same subspace size and generally fewer iterations than the base method when solving for multiple clustered eigenvalues with small dimension size. We also provide an explanation for the acceleration seen from implementing Polyak's heavy-ball method, including justifying the given parameter range.

\section{Introduction}

    Given matrices $A, B\in\mathbb{C}^{n\times n}$, the generalized eigenvalue problem is to find $\lambda\in\mathbb{C}$ and $x\in\mathbb{C}^n, x\not=0$ such that $Ax=\lambda Bx$, also known as the pencil problem for linear pencil (A, B). The main problem we consider is solving for the eigenvalues of the pencil $(A,B)$ where $A$, $B$ are symmetric, sparse, and large, and $B$ is positive definite. Such problems arise in dynamic analysis in structural engineering \cite{Grimes_1986} and electronic structure calculations \cite{Hoshi_2012}.

    Many methods exist which solve this pencil problem, such as gradient type methods which are attractive in that they do not require the computation of any matrix inverse. Steepest descent is a well known gradient type method which, while reliable, has very slow convergence. Popular gradient type methods include LOPCG and its block version LOBPCG of \cite{Knyazev_2000}, which extend the search space of steepest descent to include the previous approximate eigenvector, allowing for significantly faster convergence in practice. LOPCG and LOBPCG use the Rayleigh-Ritz procedure to choose the locally optimal steepest descent step size.

    In 2002, Golub and Ye introduced their competitive inverse-free Krylov subspace iterative method \cite{Golub_Ye_2002}, a gradient type method which solves for extreme eigenvalues of our stated problem. At each step, a Krylov subspace is formed using the current eigenvector approximation and used in the Rayleigh-Ritz projection method to output a better approximation to the pencil's algebraically smallest (or largest) eigenvalue. Small subspace size is recommended as the method takes advantage of multiple solves on much smaller pencils instead of solving the large problem directly. Notably, for Krylov subspace of dimension size 2, the method is equivalent to steepest descent with locally optimal step size and a locally optimally chosen parameter on the current iterate.
    
    The method was proven to always converge to the extreme eigenpair at least linearly. However, the algorithm may perform poorly when eigenvalues are tightly  clustered. The original paper explores accelerating the method using preconditioning, and a second paper presented by Quillen and Ye in 2010 tackles the poor convergence caused by clustered eigenvalues by introducing a block version of the method \cite{Quillen_Ye_2010}. A black box implementation of the block inverse-free Krylov subspace method called BLEIGIFP is made available by Ye, which includes the previous block of approximate eigenvectors in each step's subspace and adaptively chooses the block size number of eigenvalues to solve for. With the inclusion of the previous block, this implementation is equivalent to LOBPCG when the Krylov subspace dimension is set at 2 and block size is fixed. In this paper, we consider another means of acceleration based on the method's relation to the well known method of steepest descent.
    
    Recent works implementing extrapolation for eigenvalue problems have shown a reduction in the number of iterations at low cost \cite{Pollock_ChristianAustin_Zhu_2024, Pollock_Scott_2021, Pollock_Nigam_2022, De_Sa_2018, Rabbani_2022, Bai_2021}. \cite{Pollock_Scott_2021} in particular explores extrapolating a restarted $k$-step Arnoldi algorithm which is equivalent to Golub and Ye's method for the standard eigenvalue problem ($B=I$). A recently popular method of momentum-type extrapolation is Nesterov accelerated gradient descent, an extrapolation modification accelerating steepest descent \cite{Nesterov_1983, Walkington_2023}. Another popular method used often in accelerating steepest descent is Polyak's heavy-ball method \cite{Polyak_1964}. We consider applying acceleration in the style of the heavy-ball method to the inverse-free Krylov subspace method as in \cite{Pollock_Scott_2021} for the generalized eigenvalue problem, which we will call the Depth-1 method, as well as acceleration inspired by generalizing Nesterov accelerated steepest descent and Polyak's heavy-ball method. The accelerated methods introduce acceleration parameter $\beta_k$ at each step $k$.
    
    For all three accelerated methods, we observed faster convergence for good choices of fixed $\beta_k=\beta$ with small subspace dimension size. When the subspace dimension size is set at 2, that is, the base method is equivalent to LOBPCG, and when solving for multiple clustered eigenvalues with good choice of fixed $\beta$, the Depth-1 and Nesterov-like accelerated methods converged on average in $62\%-82\%$ of LOBPCG's total iterations, and the heavy-ball-like accelerated method converged on average in $61\%-78\%$ of LOBPCG's total iterations in our tests. When eigenvalues were tightly clustered, the accelerated methods performed substantially better than the base method, converging in $33\%$ of LOBPCG's total iterations on average. The heavy-ball-like method was in particular more stable than LOBPCG and the other accelerated methods with respect to the initial vector when solving for extremely clustered eigenvalues. When the subspace dimension size is set at 3, we also saw significant acceleration while solving for clustered eigenvalues, but when the dimension size is around 6 or higher, there is much less benefit from acceleration.

    The remainder of the paper is organized as follows. In section 2, we discuss the generic eigensolver and inverse-free Krylov subspace methods introduced by Quillen, Golub, and Ye. We also discuss LOPCG and show numerical results studying the effect of fixing one of its parameters to give motivation for our presently introduced methods. In section 3 we discuss generalizing the Nesterov and heavy-ball methods, and we present our accelerated version of the inverse-free Krylov subspace single vector method along with analysis of extrapolation. We present numerical results for the single vector method in section 4, comparing it to the original method. In section 5, we extend acceleration to the block method and present analysis proving convergence for our methods. Numerical results for the block accelerated methods are presented in section 6.

\section{Background}
    The algebraically smallest eigenvalue of the pencil $(A,B)$ can be found by minimizing the Rayleigh quotient
    \begin{align}\label{eqn:rayquot}
      \rho(x)\equiv\frac{x^TAx}{x^TBx},
    \end{align}
    for all vectors $x$ in our subspace. One such strategy in minimizing the Rayleigh quotient is to perform the Rayleigh-Ritz projection method on a subspace of our original space. The smallest eigenvalue of the subspace is obtained and given as an approximation to the smallest eigenvalue of $(A,B)$. This process may be repeated iteratively to give a new, better approximation at each step.

    \begin{algorithm}[h!]
    \caption{Generic Eigensolver}
    \textbf{Input:} Symmetric $A\in\mathbb{R}^{n\times n}$, S.P.D. $B\in\mathbb{R}^{n\times n}$.
    \begin{algorithmic}[1]
        \For{$k=0,1,2,\dots$}
            \State Construct a subspace $S^{(k)}$ of dimension $m\ll n$ from $p$ approximate eigenpairs\newline\phantom{ww}$(\rho_1^{(k)},x_1^{(k)}),\dots,(\rho_p^{(k)},x_p^{(k)})$.
            \State Perform Rayleigh-Ritz on $(A,B)$ with respect to $S^{(k)}$ to extract the $p$ desired approximate\newline\phantom{ww}eigenpairs $(\rho_1^{(k+1)},x_1^{(k+1)}),\dots,(\rho_p^{(k+1)},x_p^{(k+1)})$.
        \EndFor
    \end{algorithmic} 
    \end{algorithm}
    
    In \cite{Quillen_Ye_2010}, Quillen and Ye describe a generic scheme for solving the $p$ smallest eigenvalues of the pencil $(A,B)$, presented as Algorithm 2.1. The scheme performs the Rayleigh-Ritz projection method to gain new approximate eigenpairs $(\rho_i^{(k)},x_i^{(k)})$, ${1\leq i\leq p}$ 
    at each step, where $\rho_i^{(k)} = \rho(x_i^{(k)})$, as given by \eqref{eqn:rayquot}, and the process is repeated iteratively until convergence criteria is met. The generic eigenvalue solver is shown to converge to eigenpairs of $(A,B)$ when the Ritz vectors ${x_1^{(k)},\dots,x_p^{(k)}}$ with ${||x_i^{(k)}||_B=1}$, where $||x||_B=\sqrt{x^TBx}$, and residuals ${(A-\rho_i^{(k)}B)x_i^{(k)}}$, ${1\leq i\leq p}$ are contained in $S^{(k)}$ \cite{Quillen_Ye_2010}. Locally optimal steepest descent and LOPCG of \cite{Knyazev_2000} can be described as the generic eigensolver for the single smallest eigenvalue with subspaces $\text{span}\{x^{(k)},(A-\rho^{(k)}B)x^{(k)}\}$ and ${\text{span}\{x^{(k)},x^{(k)}-x^{(k-1)},{(A-\rho^{(k)}B)x^{(k)}}\}}$ at step $k$, respectively.

    \subsection{Inverse-free Krylov subspace method}
    
    Introduced in 2002 by Golub and Ye in \cite{Golub_Ye_2002}, the inverse-free preconditioned Krylov subspace method is a variant of the generic eigensolver using the Krylov subspace
    \[K_m\equiv \text{span}\{x^{(k)}, (A-\rho^{(k)}B)x^{(k)},\dots,(A-\rho^{(k)}B)^mx^{(k)}\}.\] For $m=1$, the method is essentially steepest descent,  with an extra parameter on $x^{(k)}$ that is chosen locally optimally at each step. In their paper, Golub and Ye proved the algorithm converges at least linearly and convergence is affected by the spectral distribution of $A-\rho^{(k)}B$. A preconditioning scheme is presented to separate the eigenvalues: a preconditioner matrix $L_k$ is chosen to give a favorable distribution for $L_k^{-1}(A-\rho_kB)L_k^{-T}$. The algorithm is then applied to the pencil
    \[(\hat{A},\hat{B})\equiv(L_k^{-1}AL_k^{-T},L_k^{-1}BL_k^{-T}),\]
    which has the same eigenvalues as $(A,B)$.  With ideal preconditioning, the convergence is quadratic. However, the ideal preconditioner $L_k=L$ from the complete $LDL^*$ factorization  of $A-\lambda_1 B$ requires high computational cost and knowledge of the very eigenvalue we are looking for. The authors recommend an using an incomplete $LDL^*$ factorization \cite{Gimenez_2003} of $A-\hat\lambda_1 B$ where $\hat\lambda_1$ is an approximation of $\lambda_1$, which requires either a lower bound on $\lambda_1$ or running a few iterations of the method without preconditioning.
    
    Golub and Ye also observed from numerical tests that the convergence rate decreases rapidly as $m$ increases. Ye provides a black-box implementation called EIGIFP, which performs competitively with the Jacobi-Davidson QR and Jacobi-Davidson QZ methods of \cite{Vandervorst_1998}  and shift-and-invert Lanczos algorithms for small $m$. However, the algorithm may perform poorly and oscillate when eigenvalues are severely clustered. A block version of the method, described next, has been introduced to handle the clustered eigenvalues.

    \subsection{Block inverse-free Krylov subspace method}
    
    In 2010, Quillen and Ye introduced a block generalization of the inverse-free Krylov subspace method of \cite{Golub_Ye_2002} along with a black-box implementation \cite{Quillen_Ye_2010}. The subspace $S^{(k)}$ used in Algorithm 2.1 at step $k$ for the block Krylov subspace method is generated as shown in \cite{Quillen_2005}: the Krylov subspaces $K_{m,1},\dots,K_{m,b}$ are generated from the current Ritz eigenpairs $\left(\rho_1^{(k)}, x_1^{(k)}\right),\dots,\left(\rho_b^{(k)}, x_b^{(k)}\right)$ and then combined as subspace 
    \[S^{(k)}=\sum_{i=1}^bK_{m,i}.\]
    $S^{(k)}$ is then orthonormalized. Preconditioning may be applied in the same manner as before.
    
    A variation is also given, where an Arnoldi-like process is used to construct a basis for a block Krylov-like subspace that differs from subspace $S$. Quillen and Ye provide BLEIGIFP, black-box implementation of this variation with the option of preconditioning. In this implementation, converged eigenvectors are deflated from the problem to avoid unnecessary computation, and adaptive choice of block size is available for when the algorithm detects a cluster size larger than the number of eigenvalues requested by the user. Inspired by LOBPCG of \cite{Knyazev_2000}, the implementation also includes the previous block of approximate eigenvalues into the current step's subspace. The authors note there is no theory proving acceleration, but it is observed in practice. This implementation performs competitively with existing algorithms and gives a good starting approximation to the next smallest eigenpair.
    
    In this paper, we are interested in accelerating the single and block Krylov subspace methods using extrapolation inspired by Polyak's heavy-ball method and Nesterov accelerated gradient descent. That is, we are interested in the effect of the inclusion of the previous block of approximate eigenpairs and extending its inclusion to the search directions $(A-\rho^{(k)}B)^jx_i^{(k)}$ for $i=1,\dots,b$, $j=1,\dots,m$. For motivation, we will look at two small examples using the LOPCG method for the single smallest eigenvector, equivalent to the block Krylov method when the previous iterate is included and Krylov parameter is fixed at $m=1$.

    \subsection{Locally Optimal Preconditioned Conjugate Gradient method}
    In this section, we introduce a variation of the subspace used in Algorithm 2.1 for steepest descent with the aim to gain understanding on the effect of adding the previous vector to steepest descent, resulting in LOPCG of \cite{Knyazev_2000}, and its correlation to Polyak's heavy-ball method of \cite{Polyak_1964}.
    
    In 1964 \cite{Polyak_1964}, Boris Polyak introduced a method using previous iterations to speed up the convergence of iterative methods, including steepest descent. Applying the heavy ball method to steepest descent gives the form
    \[x^{(k+1)}=x^{(k)}+\beta_k(x^{(k)}-x^{(k-1)})-\tau_k\nabla f(x^{(k)}),\]
    for parameters $\tau_k$ and $\beta_k$. For our choice of $f(x)$ as the Rayleigh quotient $\rho(x)$, we have
    \[x^{(k+1)}=x^{(k)}+\beta_k(x^{(k)}-x^{(k-1)})-\tau_k(A-\rho^{(k)}B)x^{(k)}.\]
    Meanwhile, LOPCG performs the Rayleigh-Ritz projection method at step $k$ using the subspace \text{span}$\{x^{(k)},x^{(k)}-x^{(k-1)},{(A-\rho^{(k)}B)x^{(k)}}\},$ which is equivalent to the iterative method \begin{equation}x^{(k+1)}=\gamma_kx^{(k)}+\beta_k(x^{(k)}-x^{(k-1)})+\tau_k{(A-\rho^{(k)}B)x^{(k)}},\end{equation} with locally optimal choice of $\gamma_k,\beta_k,\tau_k$. We can see this agrees with the heavy-ball method with an extra parameter on $x^{(k)}$. We note that the acceleration theory proven for the heavy-ball method in \cite{Polyak_1964} assumes convexity for the function being minimized \cite{Zavriev_Kostyuk_1993}, which in our case is the non-convex Rayleigh quotient.

    \begin{figure}
        \centering
        \includegraphics[width=0.49\textwidth]{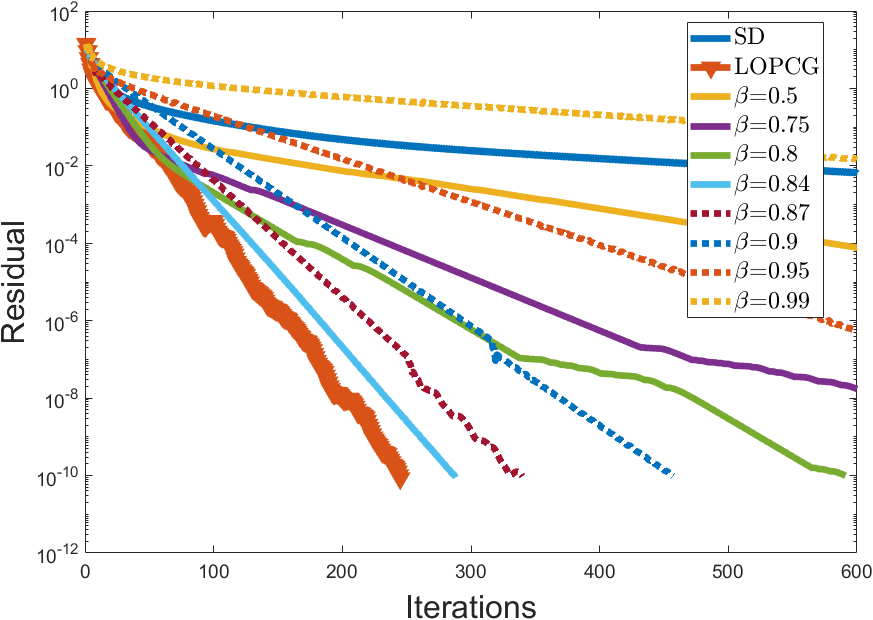}
        \includegraphics[width=0.49\textwidth]{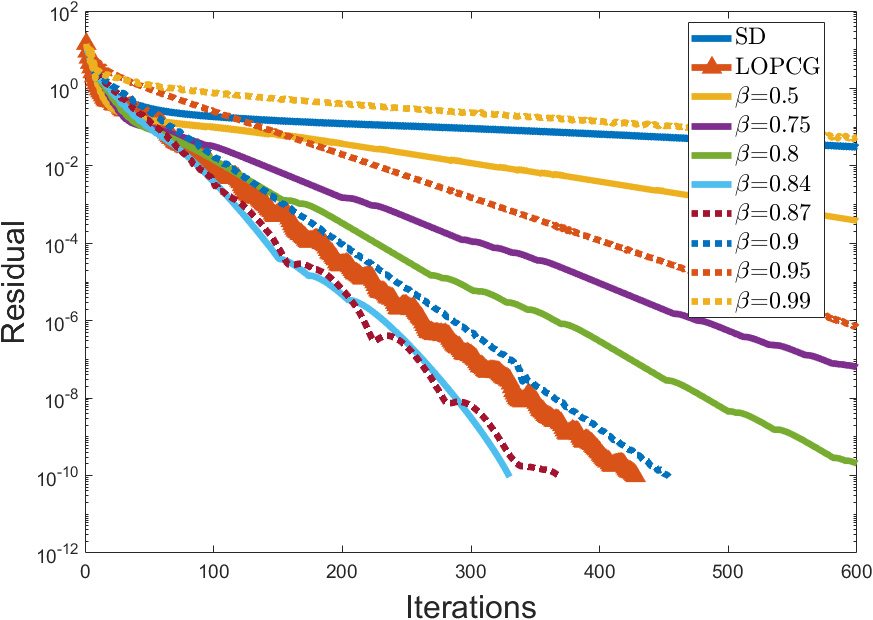}
        \caption{Residual convergence solving for the smallest eigenvalue. Left and right differ by the randomly generated initial iterate.}
    \end{figure}

    \subsubsection{Example 1}
    
    In order to study how the addition of the previous vector affects convergence while informed by the heavy-ball method, we consider using the subspace \begin{equation}\text{span}\{x^{(k)}+\beta(x^{(k)}-x^{(k-1)}),{(A-\rho^{(k)}B)x^{(k)}}\}\end{equation} with fixed parameter $\beta$ in Algorithm 2.1, which allows us to directly choose the parameter $\beta_k$ for (2) instead of letting Rayleigh Ritz implicitly choose it optimally at each step. We will call this our fixed $\beta$ method. For matrix $A$, we generated a $500\times500$ diagonal matrix with entries $(i,i)=0.1i$ for $1\leq i\leq500$ and set $B=I$. Thus, the eigenvalues of $(A,B)$ are evenly spaced with the smallest at $\lambda_1=0.1$ with the first standard basis vector $e_1$ as a corresponding eigenvector. Also, the subspace used by LOPCG is then equivalent to $\text{span}\{x^{(k)},x^{(k)}-x^{(k-1)},{Ax^{(k)}}\}$. We generate the initial iterate $x^{(0)}$ randomly with $||x^{(0)}||_2=1$ and declare convergence when the residual norm of the latest approximate eigenpair is smaller than $10^{-10}$. Figure 1 shows results for various selected $\beta$ compared to steepest descent ($\beta=0$) and LOPCG without preconditioning ($\beta$ chosen locally optimally at each step).

    We note, as $B=I$, using the subspace 
    \begin{align*}
    &\text{span}\{x^{(k)}+\beta_k(x^{(k)}-x^{(k-1)}),{(A-\rho^{(k)}B)(x^{(k)}+\beta_k(x^{(k)}-x^{(k-1)})}\}
    \\
    &= \text{span}\{x^{(k)}+\beta_k(x^{(k)}-x^{(k-1)}),{A(x^{(k)}+\beta_k(x^{(k)}-x^{(k-1)})}\}
    \end{align*}
    for this example reduces the method to the extrapolated restarted $k$-step Arnoldi algorithm of \cite{Pollock_Scott_2021}. While our subspace (3) is slightly different, we may use the same analysis of extrapolation shown in \cite{Pollock_Scott_2021} to explain convergence behavior for certain choice of $\beta$ for our example: As $A$ is symmetric, it has a basis of orthonormal eigenvectors $\{v_i\}_{i=1}^{500}$. Let $x^{(k)}$ denote the output of the fixed $\beta$ method when $\beta=0$, that is, essentially steepest descent, at step $k$ and define
    \[\eta_i^{(k+1)}:=\frac{v_i\cdot\text{proj}_{v_i}x^{(k+1)}}{v_i\cdot\text{proj}_{v_i}x^{(k)}},\] which quantifies the change in contribution of $v_i$ from $x_k$ to $x_{k+1}$. When $|\eta_i^{(k+1)}|<1$, the mode is being damped, and once $\eta_i^{(k+1)}=0$, the mode is no longer contributing to our iterate. Let $y^{(k)}$ denote the output of the fixed $\beta$ method when $\beta\not=0$ at step $k$. Then
    \[\hat{\eta}_i^{(k+1)}:=\frac{v_i\cdot\text{proj}_{v_i}y^{(k+1)}}{v_i\cdot\text{proj}_{v_i}x^{(k)}}=(1+\beta_{k+1})\eta_i^{(k+1)}-\beta_{k+1}.\]
    When $|\hat{\eta}_i^{(k+1)}|<|\eta_i^{(k+1)}|<1$, the fixed $\beta$ method is more effectively damping unwanted eigenmodes than steepest descent. As shown in Remark 3.4 of \cite{Pollock_Scott_2021}, while choosing $\beta=0.25$ or $\beta=0.5$ yields a larger interval satisfying $|\hat{\eta}_i^{(k+1)}|<|\eta_i^{(k+1)}|$ than $\beta=0.75$, the modes damped the most by $\beta=0.75$ are the ones that decay more slowly for the original method than the modes damped most by $\beta=0.25$ and $\beta=0.5$. This holds true for higher values of $\beta$, but, when close to $\beta=1$, the modes that were being the most damped by the original method are now barely being damped. We then expect to see best convergence for $\beta\in[.75,.9]$.

    \begin{figure}
        \centering
        \subfloat{\includegraphics[width=0.65\textwidth]{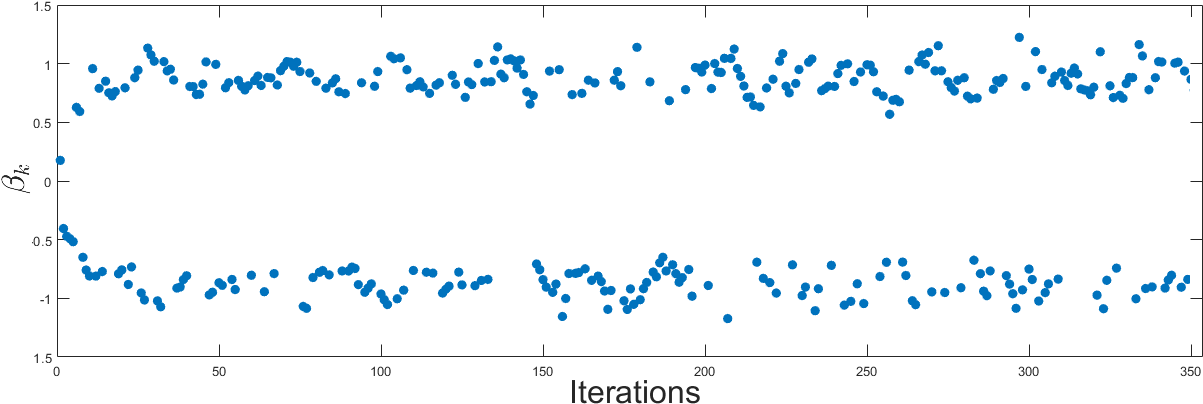} \label{fig:a}}\\
        \subfloat{\includegraphics[width=0.65\textwidth]{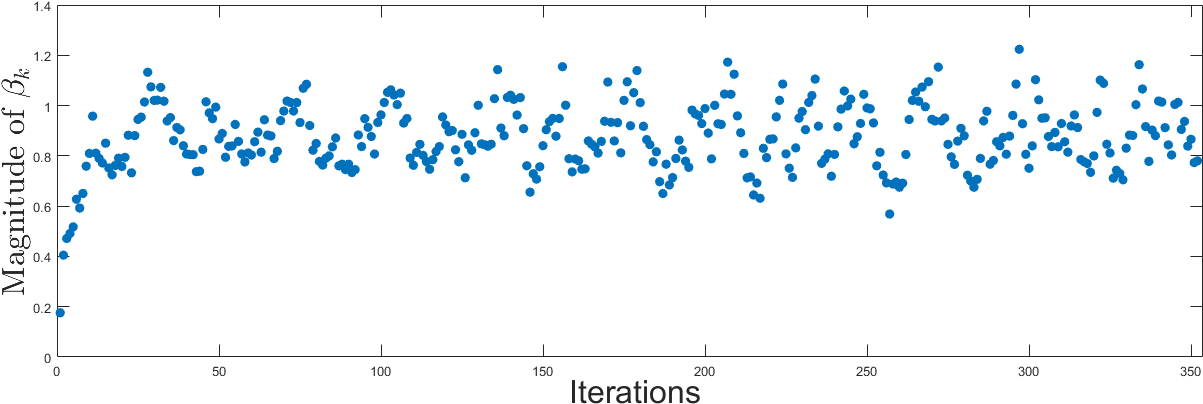} \label{fig:b}}
        \caption{The values and magnitudes of $\beta_k$ chosen by LOPCG for a randomly generated initial iterate.} \label{fig:AB}
    \end{figure}
    
    We observed the following behavior. As $\beta$ approaches to zero, the fixed $\beta$ method approached steepest descent as expected. The fixed $\beta$ method outperformed steepest descent for $0<\beta<1$, performed similarly or worse than steepest descent for $\beta\approx1$, and may fail to converge for larger $\beta$. The optimal choice of fixed $\beta$ was around $\beta\approx0.84$, which matches Polyak's suggested choice of fixed $\beta$ between $.8$ and $.99$ for convex methods \cite{Polyak_1964} as well as the analysis from \cite{Pollock_Scott_2021}. We may solve a simple system of equations using (2) to obtain LOPCG's implicit choices of $\beta_k$, shown in Figure 2 for a randomly generated intial iterative. We observed LOPCG's choice of $\beta_1$ had magnitude around 0.2 and $|\beta_k|$ quickly grew to oscillate from around 0.6 to 1.2. We also noted that while the fixed $\beta$ methods converge for $\beta_k$ values positive at each step, LOPCG's $\beta_k$ values fluctuate between positive and negative. We discovered this correlated with the method switching between converging to $e_1$ and $-e_1$, which are both solutions to the problem. Forcing the method towards $e_1$ rather than $-e_1$, that is, multiplying the approximate eigenvector by $-$1 when the first entry is negative, resulted in all positive locally optimal $\beta_k$ while achieving the same exact residual norm at each step as the method without forcing a direction.
    
    Occasionally, we observed the fixed $\beta$ method outperforming LOPCG as shown in the right graph in Figure 1. We discovered this occurs when our initial iterate had a smaller magnitude first entry compared to its other entries. We tested this by generating a ``good choice" of initial iterate, i.e. one close to $e_1$, with the first entry equal to 100 and all other entries equal to 1 before normalizing. A ``bad choice" of initial iterate was generated similarly with the first entry equal to $0.01$. Results are shown in Figure 3. We observed LOPCG significantly outperforming the fixed $\beta$ method for all $\beta$ when using the good initial iterate, while the fixed $\beta$ method outperformed LOPCG for $\beta=0.8$ to $\beta=0.95$ when using the bad initial iterate. Specifically, the total number of iterations needed for LOPCG to converge ranged widely for different choices of initial iterate, while the fixed methods had a much smaller range.

    \begin{figure}
        \centering
        \includegraphics[width=0.49\textwidth]{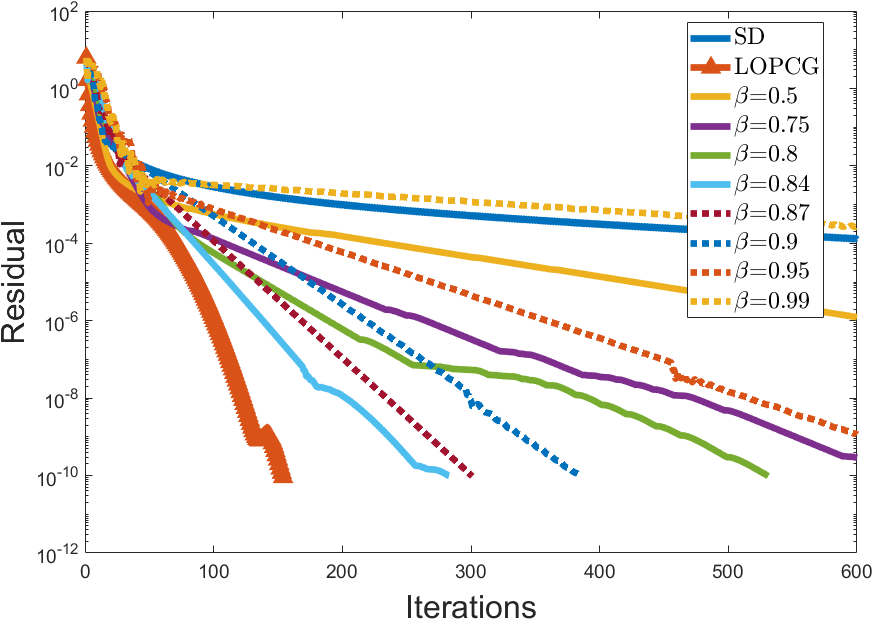}
        \includegraphics[width=0.49\textwidth]{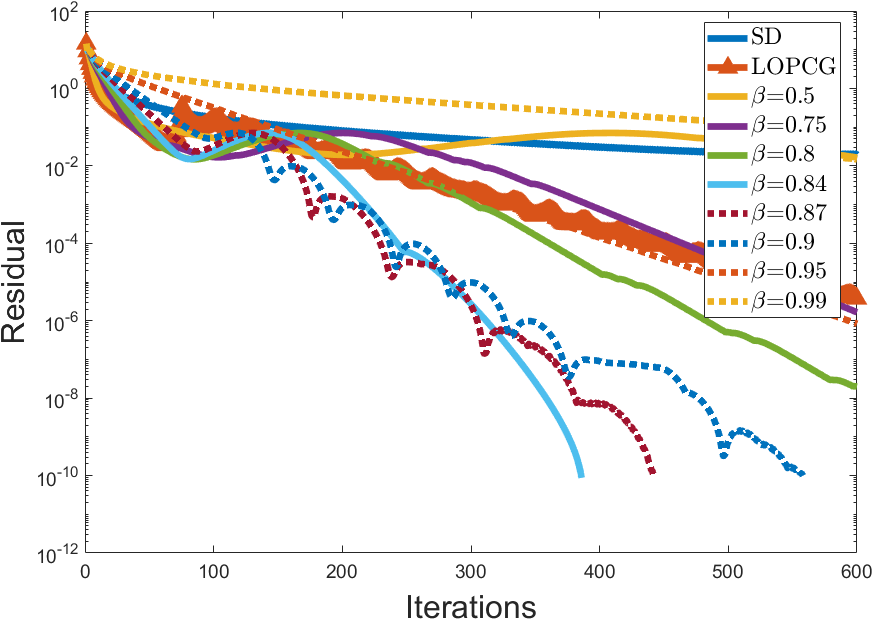}
        \caption{Residual convergence solving for the smallest eigenvalue. Left: Good choice of initial iterate. Right: Bad choice of initial iterate.}
        \includegraphics[width=0.49\textwidth]{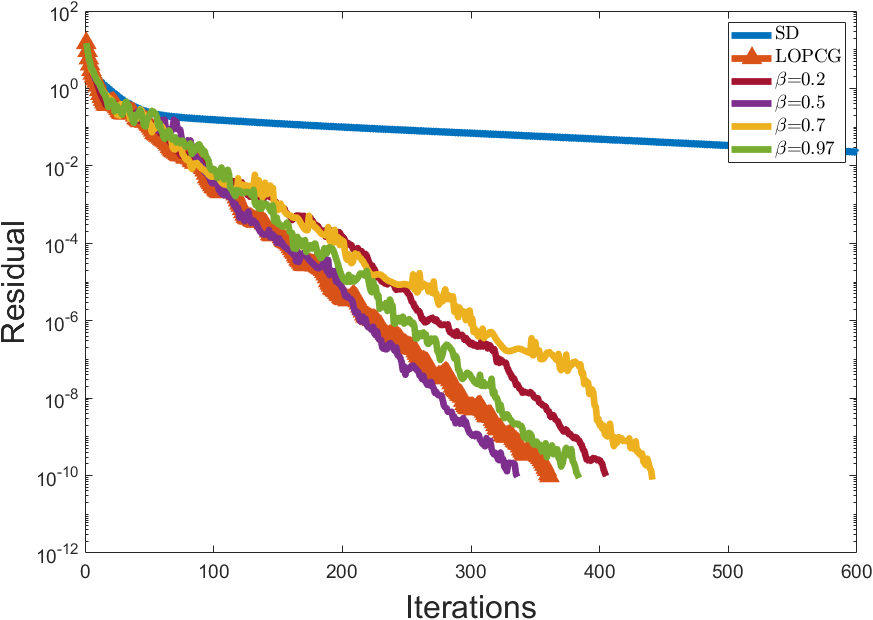}
        \includegraphics[width=0.49\textwidth]{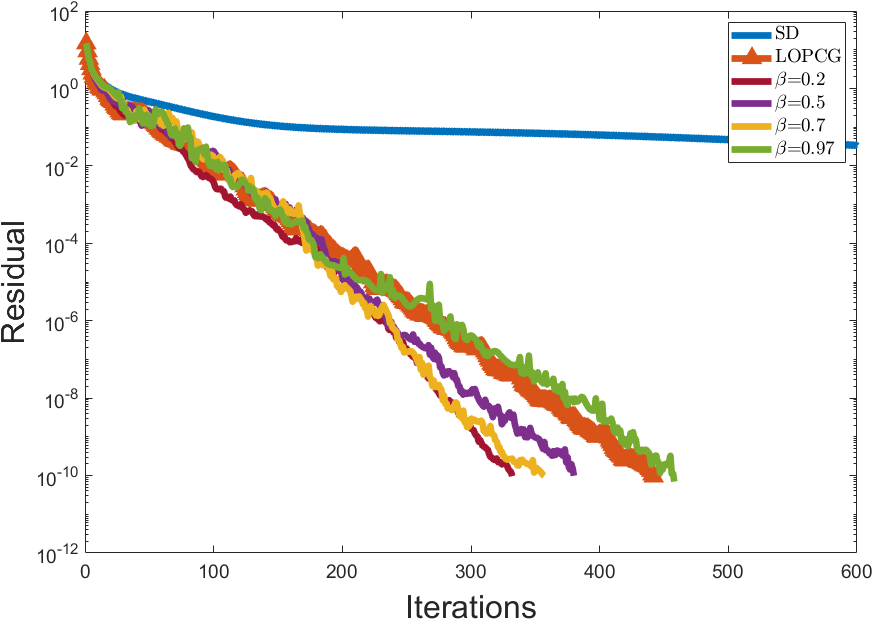}
        \caption{Residual convergence solving for the smallest eigenvalue. Left and right differ by the randomly generated initial iterate.}
    \end{figure}

    \subsubsection{Example 2}
    We next consider the subspace \[\text{span}\{x^{(k)},x^{(k-1)},{(A-\rho^{(k)}B)(x^{(k)}+\beta_k(x^{(k)}-x^{(k-1)})}\},\] which, as $B=I$, is equivalent to \[\text{span}\{x^{(k)},x^{(k-1)},{A(x^{(k)}+\beta_k(x^{(k)}-x^{(k-1)})}\}.\] This choice of subspace is motivated by by the results of \cite{Pollock_Scott_2021} and LOPCG's ability to choose $\beta_k$ locally optimally, as well as the ease in proving convergence for this choice of subspace by including $x^{(k)}$ in the subspace, as shown in section 5.1. Figure 4 shows results for two random choices of initial iterate for a few values of fixed $\beta$. We observe similar results as before, where the method occasionally outperforms LOPCG for good choice of $\beta$, especially when the choice of initial iterate is ``bad". We also observed the best performing fixed $\beta$ varied greatly, including values from $\beta=0.1$ to $\beta=.97$, unlike before. This potential speed-up motivates our interest in changing the Krylov subspace of \cite{Golub_Ye_2002} to include previous iterates in the search direction. We will consider extrapolation in the style of \cite{Pollock_Scott_2021} and that of generalized Nesterov acceleration and generalized heavy-ball.
     
    \section{Acceleration}
    Inspired by the previous results, we now consider different methods of including the previous block of iterates in the search directions. We consider the heavy-ball method as implemented in Example 2, which we call our depth-1 method, and explore extrapolation in the form of generalized Nesterov acceleration and by generalizing the heavy-ball method.
    
    \subsection{Generalized Nesterov acceleration}

    In 1983 \cite{Nesterov_1983}, Yurii Nesterov introduced a modification of the steepest descent algorithm. Given a closed, convex subset $U$ of Hilbert space $X$ and a convex function $f:X\rightarrow\mathbb{R}$ to minimize, steepest descent takes the form of
    \[x_{k+1}=x_k-\tau_k\nabla f(x_k)=:q(x_k),\]
    at step $k$ for step size $\tau_k$. Nesterov combines the current gradient at $k$ and past momentum to give the update variable
    \[y_k=x_k+\frac{\lambda_{k-1}}{\lambda_k}(x_k-x_{k-1}),\] where 
    \[\lambda_0=1,\hspace{1cm}\lambda_k=\frac{1+\sqrt{1+4\lambda_{k-1}^2}}{2},\hspace{1cm}i=1,2,\dots\]
    Nesterov accelerated steepest descent is then \[x_{k+1}=y_k-\tau_k\nabla f(y_k)=q(y_k).\] For smooth $f$, Nesterov proved the convergence rate of his accelerated method is $O(1/n^2)$, improving steepest descent's rate of $O(1/n)$. 

    In 2020 \cite{Hans_2020}, Mitchell, Ye, and de Sterck presented a generalized Nesterov extrapolation approach: given a generic iterative optimization method with update $x_{k+1}=q(x_k)$, Nesterov extrapolation may be applied to obtain 
    \[x_{k+1}=q(y_k),\hspace{1cm}y_k=x_k+\beta_{k}(x_k-x_{k-1}),\]
    where $\beta_{k}$ is our momentum weight. The authors use the generalized approach to accelerate alternating least squares, a non-convex problem. The original momentum weight sequence is shown unsuitable for their method, and they recommend instead using $\beta_k=\dfrac{||\nabla f(x_k)||}{||\nabla f(x_{k-1})||}$. The authors reason this $\beta_k$ will be large when the algorithm is converging slowly and we want to apply more acceleration, and small when the algorithm has fast convergence and we want to let it be. Similar choice and reasoning to this $\beta_k$ appears in \cite{Dallas_2024} for accelerating Newton's method.

    \subsection{Generalized heavy-ball method}
    
    We can also generalize the heavy-ball method to obtain the iteration    
    \[x_{k+1}=q(x_k)+\beta_{k+1}(x_k-x_{k-1}).\]
    We want to keep the updated approximate eigenvector as the output of our method, so we use change of variables to get the equivalent iterative method
    \[x_{k+1}=q(y_k),\hspace{1cm}y_k=x_k+\beta_k(y_{k-1}-y_{k-2}),\hspace{1cm}y_0=x_0.\]
    For ease of implementation and reduced storage, we consider the following variation as our heavy-ball-like method:
    \[x_{k+1}=q(y_k),\hspace{1cm}y_k=x_k+\beta_ky_{k-1},\hspace{1cm}y_0=x_0.\]

    \subsection{Applying acceleration to the single vector method}

    As the inverse-free Krylov subspace method needs only $x_{k-1}$ as input at step $k$, we may then write the algorithm from \cite{Golub_Ye_2002} as $x_k=q(x_{k-1})$ and consider applying generalized Nesterov or heavy-ball acceleration, as shown in Algorithm 3.1. We also consider our depth-1 extrapolation method as in \cite{Pollock_Scott_2021}. For all three methods, fixing $\beta_k$ at $0$ at every step gives the original inverse-free Krylov subspace method.
    
    \begin{algorithm}[h!]
    \caption{Accelerated Inverse-Free Krylov Subspace Method.}
    \textbf{Input:} Symmetric $A\in\mathbb{R}^{n\times n}$, S.P.D. $B\in\mathbb{R}^{n\times n}$, $x_0\in\mathbb{R}^n$ with $||x_0||=1$, and $m\geq1$.
    \begin{algorithmic}[1]
        \State $\rho_0=\rho(x_0),\enskip y_0=x_0,\enskip\theta_0=\rho_0$
        \For{$k=0,1,2,\dots$}
            \State Construct basis $Z_m$ for $K_m=\text{span}\{x_k,y_k,(A-\theta_kB)y_k,...,(A-\theta_kB)^my_k\}$
            \State Form $A_m=Z_m^T(A-\theta_kB)Z_m$ and $B_m=Z_m^TBZ_m$
            \State Find the smallest eigenpair $(\mu, v)$ of $(A_m, B_m)$
            \State $x_{k+1}=Z_mv$
            \State $\rho_{k+1}=\rho(x_{k+1})$
            \State For depth-1 extrapolation:  $y_{k+1}=x_{k+1}+\beta_k(x_{k+1}-x_k),\enskip\theta_k=\rho_k$
            \State For Nesterov-like acceleration:  $y_{k+1}=x_{k+1}+\beta_k(x_{k+1}-x_k),\enskip\theta_k=\rho(y_k)$
            \State For heavy-ball-like acceleration:  $y_{k+1}=x_{k+1}-\beta_ky_k,\enskip\theta_k=\rho_k$
        \EndFor
    \end{algorithmic} 
    \end{algorithm}

    The difference between the base method and accelerated methods is the Krylov subspace. We include the previous vector for the base method as introduced in the block method.

    \noindent{\bf Base single vector subspace:} \[K_m=\text{span}\{x_k,x_{k-1},(A-\rho_kB)x_k,\dots,(A-\rho_kB)^mx_k\}.\]
    We note the base subspace is then equivalent to the subspace $\text{span}\{x_k+\beta_k(x_k-x_{k-1}),(A-\rho_kB)x_k,\dots,(A-\rho_kB)^mx_k\}$ with the locally optimal choice of $\beta_k$ at each step. Inspired by this, we add $x_k$ to our accelerated methods' subspaces as well. In section 5.1 we show the addition of $x_k$ to the subspaces allows us to prove convergence for the block methods, which includes the single vector methods for block size 1.
    
    \noindent{\bf Single vector depth-1 extrapolation method subspace:} \[K_m=\text{span}\{x_k,y_k,(A-\rho_kB)y_k,...,(A-\rho_kB)^my_k\},\]
    where $y_{k}=x_{k}+\beta_k(x_k-x_{k-1})$.
    
    \noindent{\bf Single vector Nesterov-like accelerated method subspace:} \[K_m=\text{span}\{x_k,y_k,(A-\theta_kB)y_k,...,(A-\theta_kB)^my_k\},\]
    where $y_{k}=x_{k}+\beta_k(x_k-x_{k-1})$, $\theta_k=\rho(y_k)$.
    
    \noindent{\bf Single vector heavy-ball-like accelerated method subspace:} \[K_m=\text{span}\{x_k,y_k,(A-\rho_kB)y_k,\dots,(A-\rho_kB)^my_k\},\]
    where $y_{k}=x_{k}+\beta_k y_{k-1}=\sum_{i=0}^k(\prod_{j=0}^{k-i}\beta_j)x_i$ where $\beta_0=1$.
    
    Though we include the current step to allow Rayleigh Ritz to implicitly chose the locally optimal weight for $y_k$, we do not actually calculate the value of $\beta_k$ at step $k$ due to high computation cost. Thus, this implicit choice is not used for the updated search directions $(A-\rho_{y_k}B)y_k,...,(A-\rho_{y_k}B)^my_k$ and a good choice of parameter $\beta_k$ at each step is still necessary to facilitate good acceleration. We consider $\beta_k = \beta$ as a fixed value between 0 and 1 and adaptively at step $k$ as $\beta_k=\dfrac{||\nabla\rho(x_k)||}{||\nabla\rho(x_{k-1})||}$, chosen similarly as the Nesterov parameter recommended in \cite{Hans_2020}. We also consider a safeguarded scheme where $\beta_k=\min\left(\dfrac{||\nabla\rho(x_k)||}{||\nabla\rho(x_{k-1})||},\beta_{max}\right)$ with $\beta_{max}\in(0,1]$.
    
    \subsection{Analysis of extrapolation}
  
    We begin with the depth-1 extrapolation and Nesterov-like accelerated methods which have similar form to the method in \cite{Pollock_Scott_2021}. We base our analysis of extrapolation on Section 3.1 of \cite{Pollock_Scott_2021} as in Section 2.3. We assume matrices $A,B\in\mathbb{R}^{n\times n}$ are symmetric with $B$ positive definite and thus invertible. Then their pencil has a basis of orthonormal eigenvectors $\{v_i\}_{i=1}^{n}$. Let $x^{(k)}$ be the approximate eigenvalue output by the inverse-free Krylov subspace method at step $k$ and $y^{(k)}=x_{k}+\beta_k(x_k-x_{k-1})$. We define $\eta^{(k+1)}$ and $\hat{\eta}^{(k+1)}$ as before:
    \[\eta_i^{(k+1)}:=\frac{v_i\cdot\text{proj}_{v_i}x^{(k+1)}}{v_i\cdot\text{proj}_{v_i}x^{(k)}},\]
    \[\hat{\eta}_i^{(k+1)}:=\frac{v_i\cdot\text{proj}_{v_i}y^{(k+1)}}{v_i\cdot\text{proj}_{v_i}x^{(k)}}=(1+\beta_{k+1})\eta_i^{(k+1)}-\beta_{k+1}.\]
    We note that $\text{span}\{x_k,y_k\}=\text{span}\{x_k,x_{k-1}\}$. Then, by including $x_k$ in our subspace, the choice of $\beta_k$ affects only our search directions $(A-\theta_kB)^iy_k$, $i=1,\dots,m$. Acceleration can still be attained from applying the extrapolated vector to the search directions, as shown in numerical results in Sections 4 and 6. However, as seen in Section 2.3.2, the best choice of fixed $\beta$ is no longer necessarily in the range of $0.8$ to $0.9$. We observed convergence in fewer iterations generally for $\beta\in(0.01,0.9)$, but the best choice of $\beta$ was often smaller, with $\beta\in(.07,.03)$.

    For the heavy-ball-like accelerated method, our observations note a much smaller range of effective $\beta$ which is explained through the analysis: Let $x^{(k)}$ be as before and $y_{k}=x_{k}+\beta_k y_{k-1}$. We now have
    \[\hat{\eta}_i^{(k+1)}:=\frac{v_i\cdot\text{proj}_{v_i}y^{(k+1)}}{v_i\cdot\text{proj}_{v_i}x^{(k)}}=\eta_i^{(k+1)}+\beta_{k+1}\left(1+\sum_{l=1}^j\prod_{m=1}^l\frac{\beta_{j+1-m}}{\eta^{(j+1-m)}}\right).\]
    Unlike before, our damping parameter depends on all previous $\eta^{(k)}$ and $\beta_k$. We note as $\eta^{(k)}$ can obtain negative values, negating fixed $\beta$ here produces similar results as postive $\beta$, unlike the previous methods. When assuming the previous $\eta^{(k)}$ is slightly larger than $\eta^{(k+1)}$, we see similar but slightly better damping for $\beta$ close to 1 compared to the damping by the previous methods. However, when $\beta>0.3$, the values of $\hat{\eta}$ no longer stay inside the range $(-1,1)$. Rather, the previously smallest eigenmodes instead become very large. This is demonstrated in our numerical results, where we observed good acceleration for the heavy-ball-like method for $\beta\in(.05,.25)$, worse convergence than the original method for $\beta>.5$, and often failure to converge for $\beta>.9$.

\section{Numerical results}

    In this section, we present results demonstrating the convergence of our accelerated inverse-free Krylov subspace methods. We consider the Laplacian eigenvalue problem with a barbell shaped domain that cause the eigenvalues of the resulting pencil to be clustered. As the inverse-free Krylov subspace method and the accelerated methods differ only by subspace, this paper uses a Matlab implementation of the generic eigensolver shown in Algorithm 2.1 with the appropriately constructed subspaces for each method. We fix the Krylov subspace size $m$ and compare the total number of iterations (referred to in \cite{Golub_Ye_2002} and \cite{Quillen_Ye_2010} as ``outer iterations"). The Matlab command {\tt eig} was used to compute the smallest eigenvalue of the $m\times m$ projected problem. All tests were performed in Matlab R2024b running on a machine with an AMS Ryzen Quadcore processor and 16 GB of RAM running Windows 10.
    
    \begin{figure}
        \centering
        \includegraphics[width=0.45\textwidth]{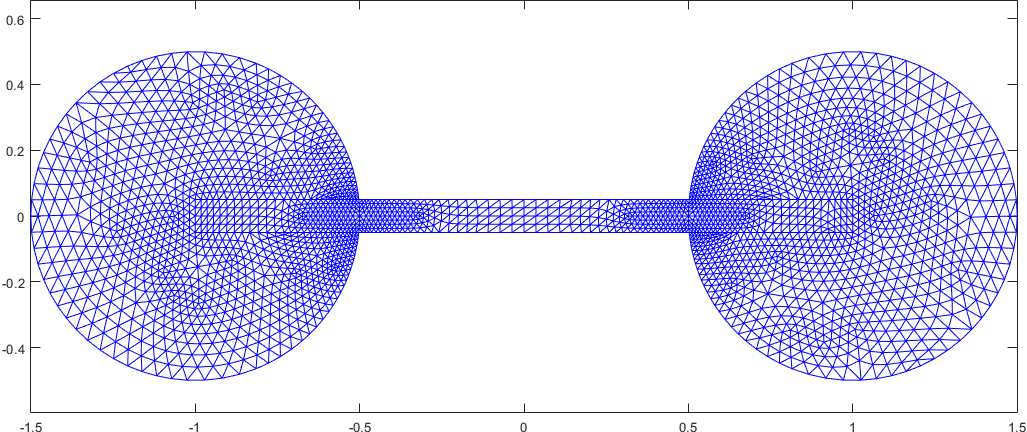}
        \caption{Mesh for the barbell shaped domain}
    \end{figure}

    \noindent We compare our method to the original inverse-free Krylov subspace method we are accelerating. We consider the Laplacian eigenvalue problem
    \begin{align}\label{eqn:laplace}
    -\Delta u(x)=\lambda u(x),\quad x\in\Omega,
    \end{align}
    with Dirichlet boundary conditions where $\Omega$ is a barbell shaped domain. This problem is the same used in \cite{Quillen_Ye_2010} to show results for BLEIGIFP. Discretizing the problem gives us stiffness matrix A and positive-definite mass matrix B, symmetric matrices of dimension $2441\times2441$ whose pencil results in eigenvalue clusters of size two. We solve for the smallest eigenvalue of the pencil.

    We use the same fixed Krylov parameter $m$ and randomly generated initial eigenvector approximation for each method. We declare convergence when the residual norm of the latest approximate eigenpair is smaller than $10^{-7}$. Results are shown in Figure 6 for the three accelerated methods for some hand-picked $\beta$ with the best performance and safeguarded adaptive $\beta_k$.
    
    \begin{figure}
        \centering
        \includegraphics[width=0.49\textwidth]{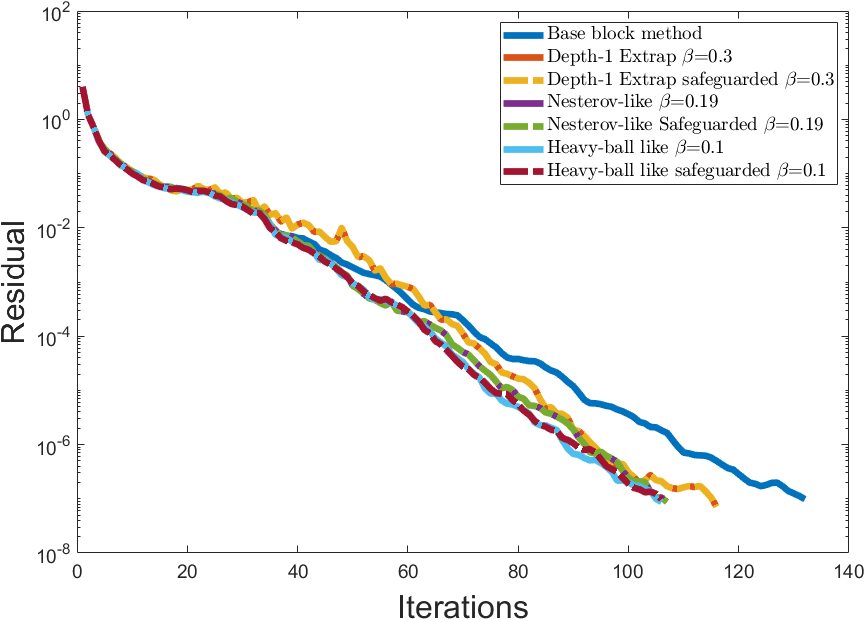}
        \includegraphics[width=0.49\textwidth]{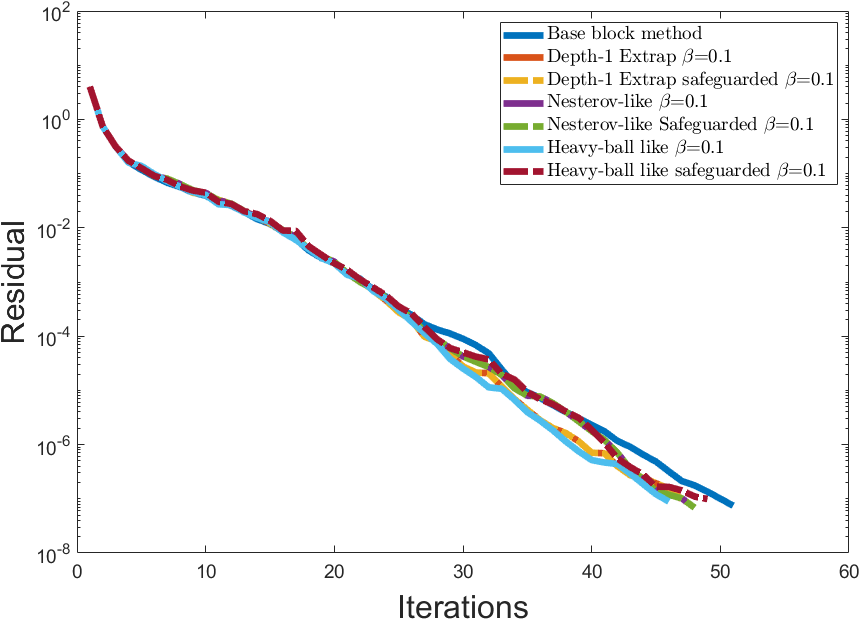}
        \caption{Residual convergence solving for the smallest eigenvalue. Left: $m=2$, right: $m=4$.}
    \end{figure}

    We observed the following behavior. The depth-1 extrapolation and Nesterov-like fixed methods performed best for $\beta<0.5$. For the heavy-ball-like method, the best seen choices for $\beta$ were around $\beta=0.1$. Often, the best performing choice of $\beta$ was so small that the fixed and safeguarded methods were the same. For small Krylov parameter $m$, the accelerated methods converged in fewer steps than the original method for well chosen $\beta$. As we increased $m$, we observed less of an impact from acceleration with the methods performing about the same as the block method for well chosen fixed $\beta$. 
    
    We note at each step the depth-1 extrapolation and heavy-ball-like methods require an extra addition and multiplication, while the Nesterov-like accelerated method requires an extra addition, multiplication, and computation of a Rayleigh quotient. The adaptive versions of the method also require computation of $\beta_k=\dfrac{||\nabla\rho(x^{(k)})||}{||\nabla\rho(x^{(k-1)})||}$.

\section{Applying acceleration to the block vector method}

    We extend applying depth-1 extrapolation, generalized Nesterov acceleration, and generalized heavy-ball acceleration to the block version of the inverse-free Krylov subspace method similarly as the single vector version, as shown in Algorithm 5.1. The block method solves for the $b$ smallest eigenvalues of $(A,B)$. Again, fixing $\beta_k$ at 0 at every step gives the original block method, and the difference between the methods is the subspace used. We use $\theta_i^{(k)}$ to represent entry $(i,i)$ of $\Theta^{(k)}$, the approxiation of the $i^{th}$ smallest eigenvalue at step $k$.

    \begin{algorithm}[h!]
    \caption{Accelerated Block Inverse-Free Krylov Subspace Method.}
    \textbf{Input:} Symmetric $A\in\mathbb{R}^{n\times n}$, S.P.D. $B\in\mathbb{R}^{n\times n}$, $X^{(0)}\in\mathbb{R}^{n\times b}$ with $X^{(0)*}BX^{(0)}=I_b$, and $m\geq1$.
    \begin{algorithmic}[1]
        \State $P^{(0)}=\text{diag}(X^{(0)*}AX^{(0)}),\enskip Y^{(0)}=X^{(0)}$, $\Theta^{(0)}=P^{(0)}$
        \For{$k=0,1,2,\dots$}
            \For{$i=0,1,\dots,b$}
                \State Construct basis $\hat{Z_i}$ of $K_m(A-\theta_i^{(k)}B,y_i^{(k)})$
            \EndFor
            \State Orthonormalize $(X^{(k)},\hat{Z_1},\hat{Z_2},\dots,\hat{Z_p})$ to obtain Z
            \State Form $A_m=Z_m^TAZ_m$ and $B_m=Z_m^TBZ_m$
            \State Find the $b$ smallest eigenpairs $(\mu_i, v_i)$, $1\leq i\leq b$ of $(A_m, B_m)$
            \State $X^{(k+1)}=ZV$, $V=(v_1,v_2,\dots,v_b)$
            \State $P^{(k+1)}=\text{diag}(\mu_1,\mu_2,\dots\mu_b)$
            \State For depth-1 extrapolation:  $Y^{(k+1)}=X^{(k+1)}+\beta_k(X^{(k+1)}-X^{(k)}),\enskip\Theta^{(k)}=P^{(k)}$
            \State For Nesterov-like acceleration:  $Y^{(k+1)}=X^{(k+1)}+\beta_k(X^{(k+1)}-X^{(k)}),\newline\phantom{For Nesterov-like acceleration::}\Theta^{(k)}=\text{diag}\left(\rho(y_1^{(k)}),\rho(y_2^{(k)}),\dots,\rho(y_b^{(k)})\right)$
            \State For heavy-ball-like acceleration:  $Y^{(k+1)}=X^{(k+1)}-\beta_kY^{(k)},\enskip\Theta^{(k)}=P^{(k)}$
        \EndFor
    \end{algorithmic} 
    \end{algorithm}

    We consider fixed values of $\beta$ and an adaptive choice of $\beta_k=\dfrac{\left|\left|\nabla\rho\left(x_1^{(k)}\right)\right|\right|}{\left|\left|\nabla\rho\left(x_1^{(k-1)}\right)\right|\right|}$, similar to the single vector methods. When solving for $b$ smallest eigenvalues, at each step, the fixed depth-1 extrapolation and heavy-ball-like methods require $b$ extra additions and multiplications, while the Nesterov-like method also requires computation of $b$ many Rayleigh quotients. The adaptive versions of the block methods require computation of $\beta_k$ at each step.

    \subsection{Analysis of convergence}

    We use Lemma 2.1.1 of \cite{Quillen_2005}, stated here as Lemma 1:

    \begin{lemma}
        Let $\lambda_1\leq\dots\leq\lambda_n$ denote the eigenvalues of $(A,B)$. Let the Ritz values obtained from Algorithm 2.1 at the $k^{th}$ iteration be denoted by $\rho_1^{(k+1)}\leq\dots\leq\rho_b^{(k+1)}$ with corresponding Ritz vectors $x_1^{(k+1)},\dots,x_b^{(k+1)}$. Suppose further that \[\{x_1^{(k)},\dots,x_b^{(k)}\}\subset S^{(k)}.\] The following relations hold
        \[\lambda_i\leq\rho_i^{(k+1)}\leq\rho_i^{(k)},\quad1\leq i\leq b\]
        \[(A-\rho_i^{(k)}B)x_i^{(k)}\perp S^{(k)},\quad1\leq i\leq b.\]
    \end{lemma}
    
    With the inclusion of the current step, we are able to use Lemma 1 to conclude the sequences $\mu_1^{(k)},\mu_2^{(k)},\dots,\mu_b^{(k)},$ converge for our accelerated methods. We follow Lemma 2.1.2 of \cite{Quillen_2005} with weakened assumptions to show these sequences do in fact converge to eigenpairs of $(A,B)$ for each accelerated method.

    \begin{theorem}
        Let $\lambda_1\leq\dots\leq\lambda_n$ denote the eigenvalues of $(A,B)$. Let the Ritz values obtained by the generic eigensolver at step $k$ be denoted $\rho_1^{(k)}\leq\dots\leq\rho_b^{(k)}$, with corresponding Ritz vectors $x_1^{(k)},\dots,x_b^{(k)},$ scaled such that $||x_i^{(k)}||_B=1$, for $1\leq i\leq b$. 
        Suppose that \[\{x_1^{(k)},\dots,x_b^{(k)},(A-\theta_1^{(k)}B)y_1^{(k)},\dots,(A-\theta_b^{(k)}B)y_b^{(k)}\}\subseteq S^{(k)},\] for some sequences $y_1^{(k)},\dots,y_b^{(k)}$ and $\theta_1^{(k)},\dots,\theta_b^{(k)}$. If, for $i=1,\dots,b$, each $y_i^{(k)}$ contains a subsequence $y_i^{(n_k)}$ converging to the same limit as a convergent subsequence of $x_i^{(k)}$, up to some non-zero scalar, and each $\lim\theta_i^{(n_k)}=\lim\rho_i^{(k)}$, then each $\rho_i^{(k)}$ converges to some eigenvalue $\hat{\lambda}$ of $(A,B)$ and $||(A-\hat{\lambda}B)x_i^{(k)}||\rightarrow0$.
    \end{theorem}

    Lemma 2.1.2 of \cite{Quillen_2005} originally assumes \[\{x_1^{(k)},\dots,x_b^{(k)},(A-\rho_1^{(k)}B)x_1^{(k)},\dots,(A-\rho_b^{(k)}B)x_b^{(k)}\}\subseteq S^{(k)}.\] We show convergence holds when we replace the search direction $(A-\rho_i^{(k)}B)x_i^{(k)}$ with ${(A-\theta_i^{(k)}B)y_i^{(k)}}$ for sequences $\theta_i^{(k)}$ with $\lim\theta_i^{(n_k)}=\lim\rho_i^{(k)}$ and $y_i^{(k)}$ with a subsequence converging to the same limit as a subsequence of $x_i^{(k)}$. Other methods resulting in sequences meeting this criteria may be explored as future research.
    
    \begin{proof}
        By Lemma 2.1.1 of \cite{Quillen_2005}, each sequence $\rho_i^{(k)}$ is convergent, say $\lim\rho_i^{(k)}=\hat{\lambda}_i$ for $i=1,\dots,b$. We follow the proof for Lemma 2.1.2 of \cite{Quillen_2005}, which follows the proof of Theorem 3.2 of \cite{Golub_Ye_2002} for the single vector method. As each sequence $x_i^{(k)}$ is bounded, each has a convergent subsequence $x_i^{(n_k)}$. Let $\lim x_i^{(n_k)}=\hat{x_i}$, $1\leq i\leq b$. By our assumption, each sequence $y_i^{(k)}$ contains a convergent subsequence, say $y_i^{(m_k)}$, where $\lim y_i{(m_k)}=c_i\hat{x_i}$, $1\leq i\leq b$, for some non-zero $c_i\in\mathbb{R}$.
        
        Fix $j$ such that $1\leq j\leq b$. Let $\hat{r}=(A-\hat{\lambda}B)\hat{y_j}=c_j(A-\hat{\lambda}B)\hat{x_j}$. Suppose $\hat{r}\not=0$, otherwise we are done. Following Lemma 2.1.2 by considering the projection $(\hat{A},\hat{B})$ of $(A,B)$ onto $\{\hat{x}_1,\cdots,\hat{x}_b,\hat{r}\}$ and assuming without loss of generality that $\hat{\lambda}_1<\dots<\hat{\lambda}_j$, we obtain the matrix
        \[\hat{A}-\hat{\lambda}_j\hat{B}=
        \left(\begin{matrix}\hat{\lambda}_1-\hat{\lambda}j & & & & * \\
        & \ddots & & & \vdots\\
        & & \hat{\lambda}_{j-1}-\hat{\lambda}_j & & *\\
        & & & 0 & \frac{1}{c}\hat{r}^T\hat{r}\\
        * & \dots & * & \frac{1}{c}\hat{r}^T\hat{r} & \hat{r}^T(A-\hat{\lambda}B)\hat{r}\end{matrix}\right),\]
        which has $j$ negative eigenvalues and conclude $\hat{\lambda}_j$ is strictly greater than the $j^{th}$ smallest eigenvalue of $(\hat{A},\hat{B})$, say $\tilde{\lambda}_j$.
        
        Let $r_j^{(k)}=(A-\theta_j^{(k)}B)y_j^{(k)}$ for all $k$, and let $(\hat{A}^{(k)},\hat{B}^{(k)})$ be the projection of $(A,B)$ onto $\{x_1^{(k)},\cdots,x_b^{(k)},\hat{r}_j^{(k)}\}$. We then have $(\hat{A}^{(m_k)},\hat{B}^{(m_k)})\rightarrow(\hat{A},\hat{B})$ as $k\rightarrow\infty$ and let $\hat{\rho}_j^{(m_k+1)}$ denote the $j^{th}$ eigenvalue of $(\hat{A}^{(m_k)},\hat{B}^{(m_k)})$. Then, as $\{x_1^{(k)},\cdots,x_b^{(k)},r_j^{(k)}\}\subseteq S^{(k)}$, we note that $\rho_j^{(k+1)}\leq\hat{\rho}_j^{(m_k+1)}$. Then
        \[\hat{\lambda}_j=\lim\rho_j^{(k+1)}\leq\lim\hat{\rho}_j^{(m_k+1)}=\tilde{\lambda_j},\]
        which contradicts our previous conclusion. Thus, $\hat{r}=0$ and so $(\hat{\lambda},\hat{x})$ is an eigenpair of $(A,B)$.
        
        The proof of $||(A-\hat{\lambda}B)x_k||\rightarrow0$ is the same as in Theorem 3.2 of \cite{Golub_Ye_2002}.
    \end{proof}
    We now show assumptions hold for each of our proposed accelerations for bounded $\beta_k$, noting that $x_k$ is bounded.
    
    \subsubsection{Depth-1 extrapolation}
    For Depth-1 extrapolation, we have $y_{k}=x_{k}+\beta_k(x_k-x_{k-1})$, and $\theta_k=\rho_k$.  As $x_k$ and $\beta_k$ are bounded, $y_{k}$ is bounded and thus has a convergent subsequence $\lim y_{n_k}$. Let $\lim y_{n_k}=\hat{y}$. $x_k$ is bounded, so $x_{n_k}$ has a convergent subsequence $x_{m_{n_k}}$. Let $\lim x_{m_{n_k}}=\hat{x}$. \[\hat{y}=\lim_{k\to\infty}y_{m_{n_k}}=\lim_{k\to\infty}x_{m_{n_k}}+\beta_k(x_{m_{n_k}}-x_{m_{n_k}-1})=\hat{x}.\] Trivially, $\lim_{k\to\infty}\theta_{n_k}=\lim_{k\to\infty}\rho_k$. Thus Depth-1 extrapolation meets the criteria of Theorem 1.
    \subsubsection{Nesterov-like}
    For Nesterov-like acceleration, we have $y_{k}=x_{k}+\beta_k(x_k-x_{k-1})$ and $\theta_k=\rho(y_k)$. We have shown this choice of $y_k$ meets our criteria and now must show $\lim\theta_{n_k}=\lim\rho_k$. Using the convergence of $y_{n_k}$, we have
    \[\lim_{k\to\infty}\theta_{n_k}=\lim_{k\to\infty}\rho(y_{n_k})=\rho(\hat{x})=\lim_{k\to\infty}\rho_k.\]
    Thus Nesterov-like acceleration meets the necessary criteria.
    \subsubsection{Heavy Ball-like}
    For heavy-ball-like acceleration, we have $y_{k}=x_{k}+\beta_ky_{k-1}$ and $\theta_k=\rho_k$. We additionally assume $\beta_k$ converges to some $\beta\not=1$ and $|\beta_k|<1$ for all $k$. Then $y_k$ is bounded and has a convergent subsequence $y_{n_k}$. Let $\lim_{k\to\infty}y_{n_k}=\hat{y}$. $x_k$ is bounded, so $x_{n_k}$ has a convergent subsequence $x_{m_{n_k}}$. Let $\lim x_{m_{n_k}}=\hat{x}$.
    \[\hat{y}=\lim_{k\to\infty}y_{m_{n_k}}=\lim_{k\to\infty}x_{m_{n_k}}+\beta_k y_{m_{n_k}-1}=\hat{x}+\beta\hat{y}\]
    Then $\hat{y}=\frac{1}{1-\beta}\hat{x}$ and heavy-ball-like acceleration meets the necessary criteria.
    
\section{Numerical results for the block version}
    \subsection{Example 1}

    \begin{figure}
        \centering
        \includegraphics[width=0.49\textwidth]{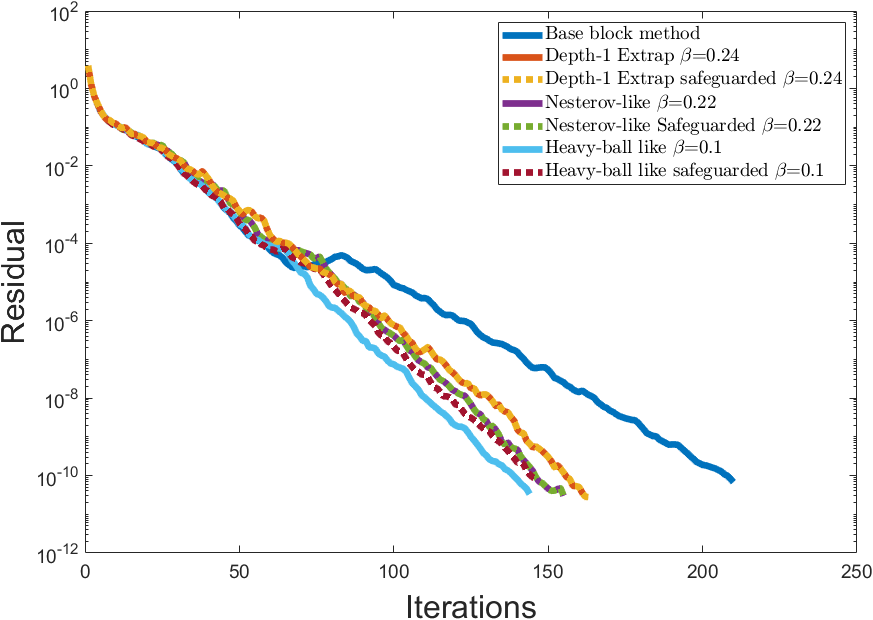}
        \includegraphics[width=0.49\textwidth]{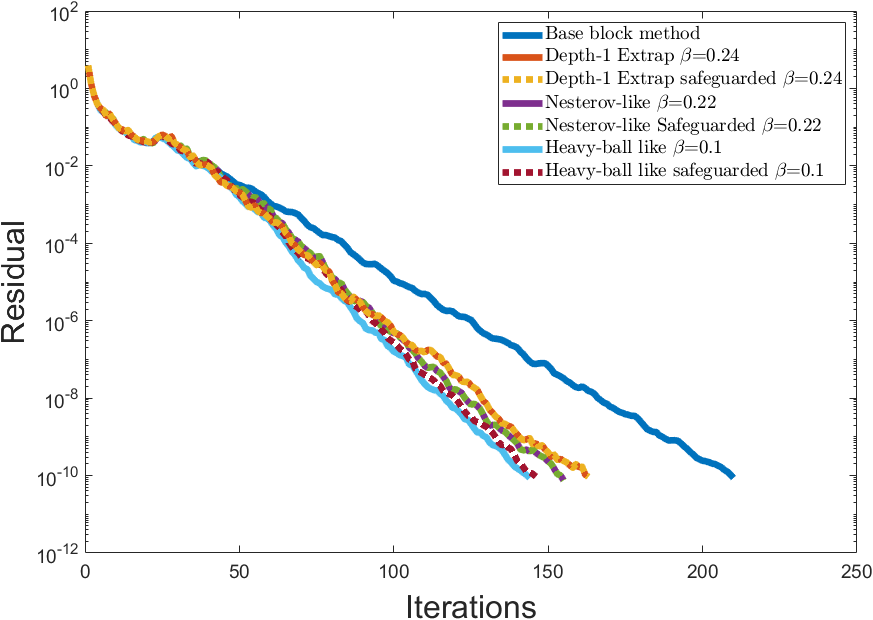}
        \caption{Residual convergence solving for the two smallest eigenvalues, $m=2$. Left: $\lambda_1$, right: $\lambda_2$.}
    \end{figure}
    
    \begin{figure}[!h]
        \centering
        \includegraphics[width=0.49\textwidth]{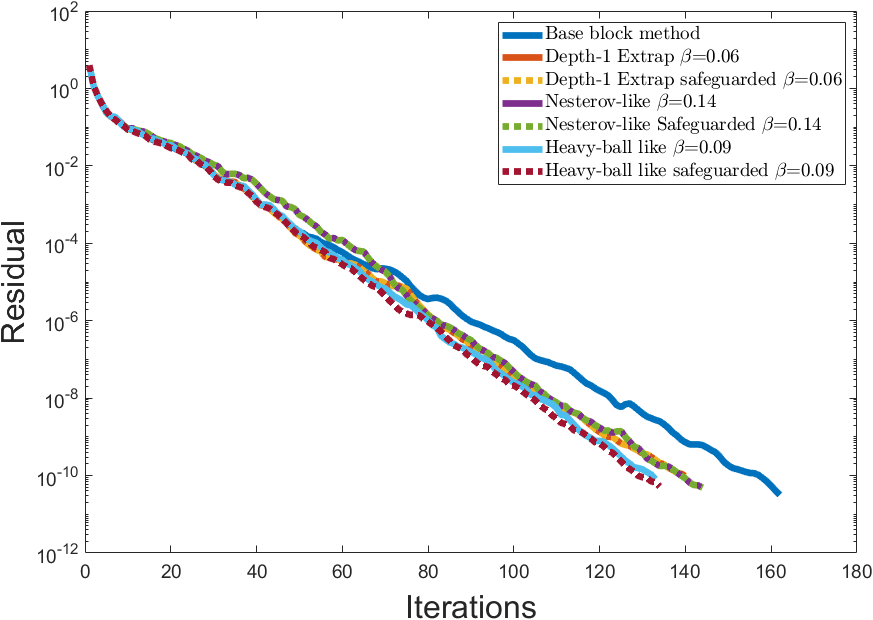}
        \includegraphics[width=0.49\textwidth]{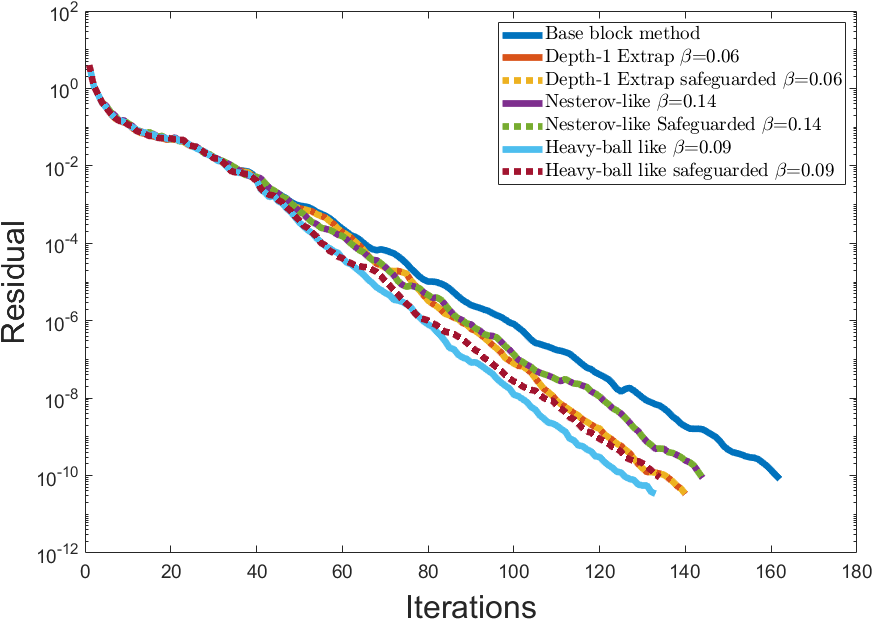}
        \caption{Residual convergence solving for the two smallest eigenvalues, $m=5$. Left: $\lambda_1$, right: $\lambda_2$.}
    \end{figure}

    We use the previous Laplacian problem \eqref{eqn:laplace} with a barbell mesh from section 4 to demonstrate the convergence of the accelerated block methods compared to the original block inverse-free Krylov subspace method. We use the same Matlab general implementation for all methods, using the appropriate subspace for each, and use the same randomly generated initial iterate for each method. We compare total number of iterations needed to converge to the smallest two eigenpairs within a tolerance of $10^{-10}$ and show convergence results for the two smallest eigenvalues for $m=2$ and $m=5$. $\beta_k$ values were hand-picked optimally for the fixed methods. Results are shown in Figure 8 for $m=2$ and in Figure 9 for $m=5$. The figures show the residual convergence of the smallest two eigenvalues of the problem.
    
    We observed the following behavior. For the various choices of $m$, the best hand-picked choice for Nesterov-like acceleration was usually around $\beta=0.25$ and for heavy-ball-like acceleration stayed smaller around $\beta=0.1$, while the choice for the depth-1 extrapolation method ranged from $\beta=0.06$ to around $\beta=0.25$. We observed good acceleration for the following ranges for the methods: $\beta<0.6$ for depth-1 extrapolation and Nesterov-like and $\beta<0.3$ for heavy-ball-like. Again, the best performing $\beta$ was so small that the fixed and safeguarded methods were often the same for the depth-1 extrapolation and Nesterov-like methods. Acceleration was more significant for the block method than when solving for a single eigenpair. For $m=2$, the accelerated methods converged in about two-thirds as many iterations than the original method well chosen $\beta$. As the Krylov subspace size $m$ increased, we saw less effect from accelerating the original method, with the accelerated methods converging in around 80\% to 100\% of the total number of iterations than the original method for $m=5$.

    \subsection{Example 2}

    \begin{figure}
        \centering
        \includegraphics[width=0.45\textwidth]{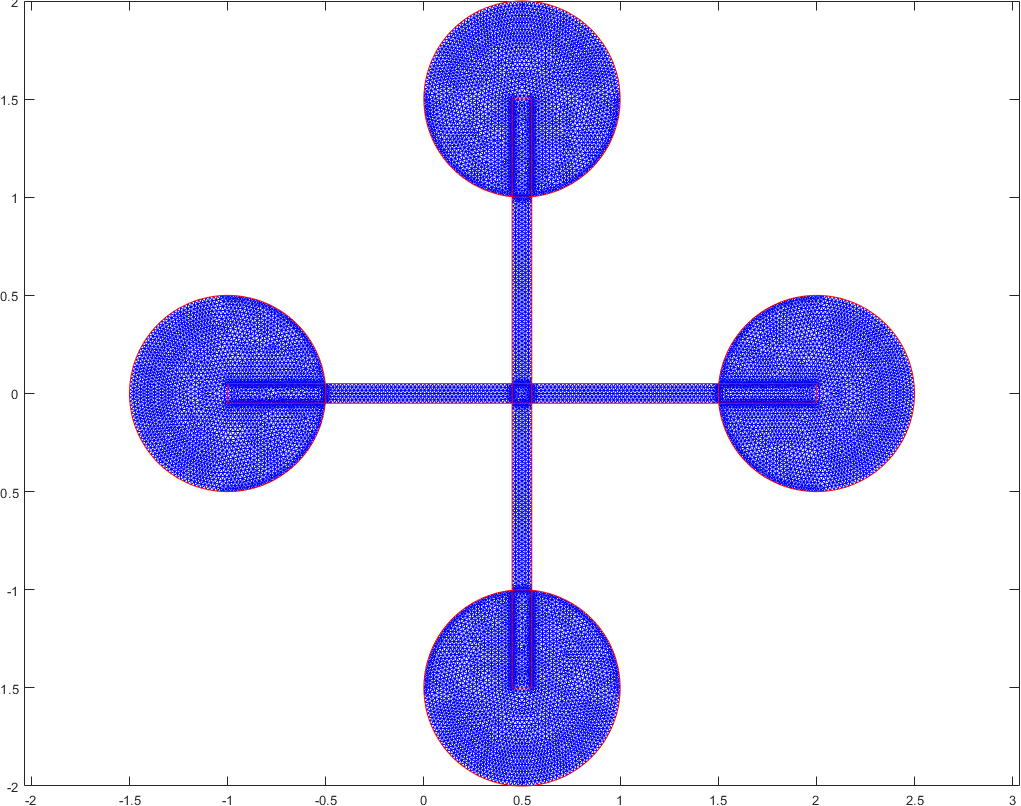}
        \caption{Mesh for the four barbell shaped domain}
    \end{figure}
    \begin{figure}
        \centering
        \includegraphics[width=0.49\textwidth]{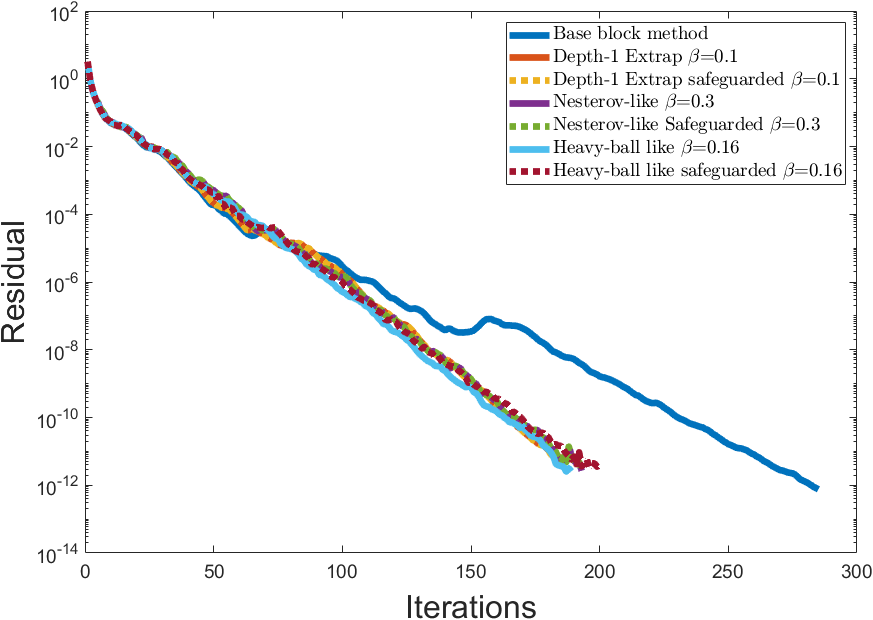}
        \includegraphics[width=0.49\textwidth]{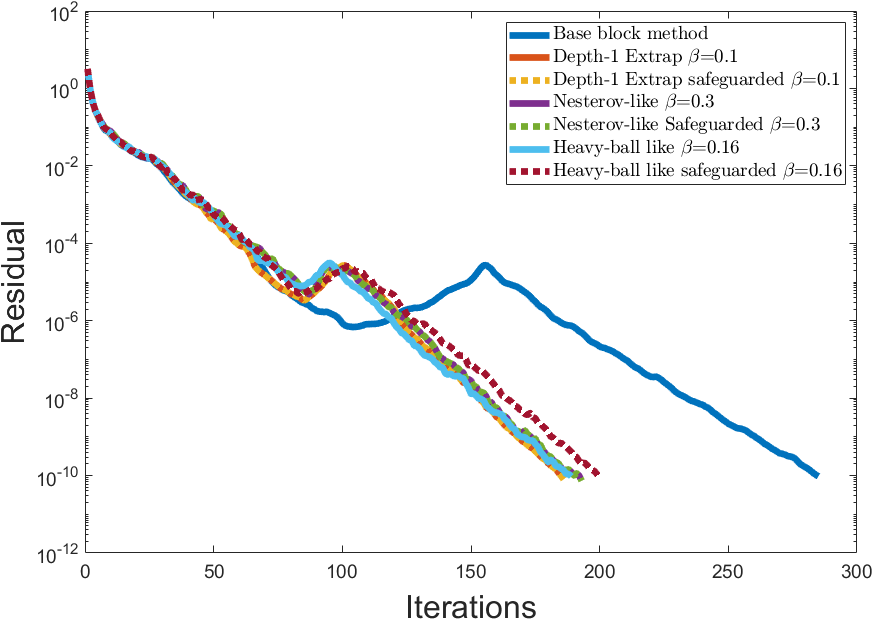}
        \includegraphics[width=0.49\textwidth]{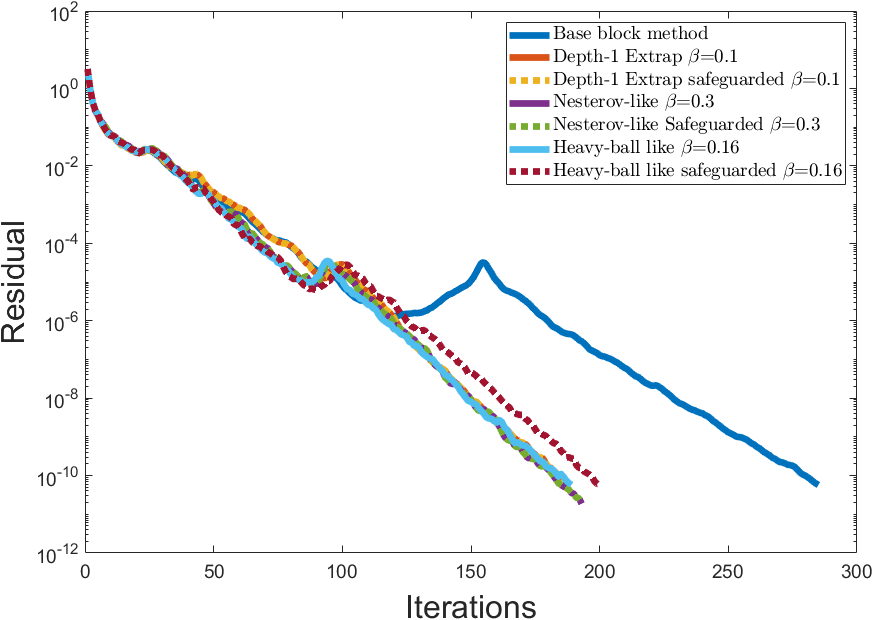}
        \includegraphics[width=0.49\textwidth]{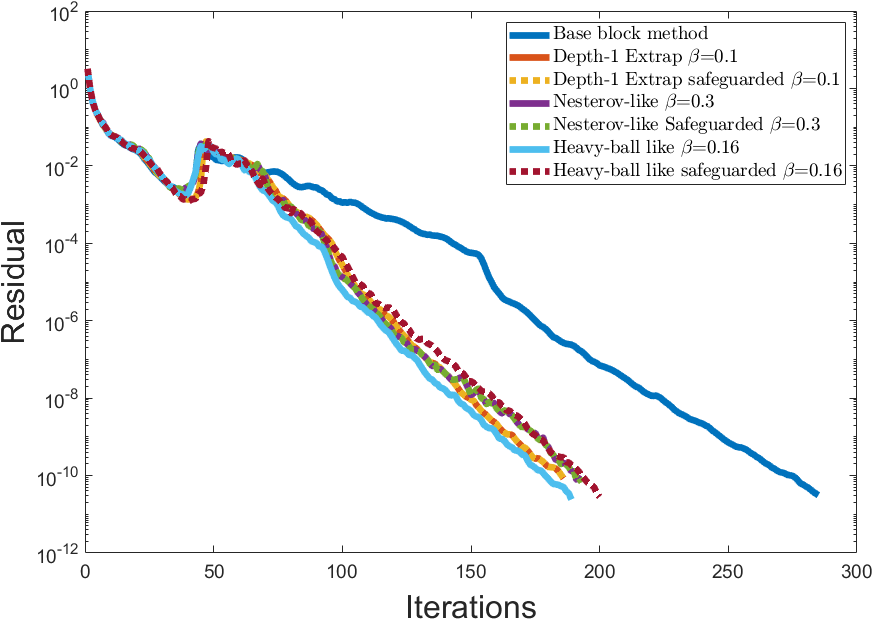}
        \caption{Residual convergence solving for the four smallest eigenvalues, $m=2$. Top left: $\lambda_1$, top right: $\lambda_2$, bottom left: $\lambda_3$, bottom right: $\lambda_4$.}
    \end{figure}

     \noindent Here we consider the Laplace eigenvalue problem \eqref{eqn:laplace} with a four barbell shaped domain, shown in Figure 9. This results in matrices of dimension $11549\times11549$ whose pencil has eigenvalue clusters of size four. For each method, we use the same random initial iterate and solve for the four smallest eigenvalues. We again choose the best $\beta$ seen for the fixed methods. Results comparing the accelerated block methods to the base block method are shown in Figure 10 for $m=2$ and Figure 11 for $m=5$. The figures show the residual convergence of the four smallest eigenvalues.
     
     For each method, we observed similar ranges for the various choices of $\beta$ as the previous example, including which $\beta$ values performed best. For $m=2$, we observe similar results as the first barbell problem with the accelerated methods converging in about two-thirds as many iterations as the original method for similar choices of $\beta_k$. For $m=5$, we observed more effective acceleration than the previous example, with the accelerated methods converging in around 75\% to 85\% of the total iterations for the original method for good choices of $\beta$. The accelerated methods appear to be more effective as the number of clustered eigenvalues we are solving for increases.

    \begin{figure}
        \centering
        \includegraphics[width=0.49\textwidth]{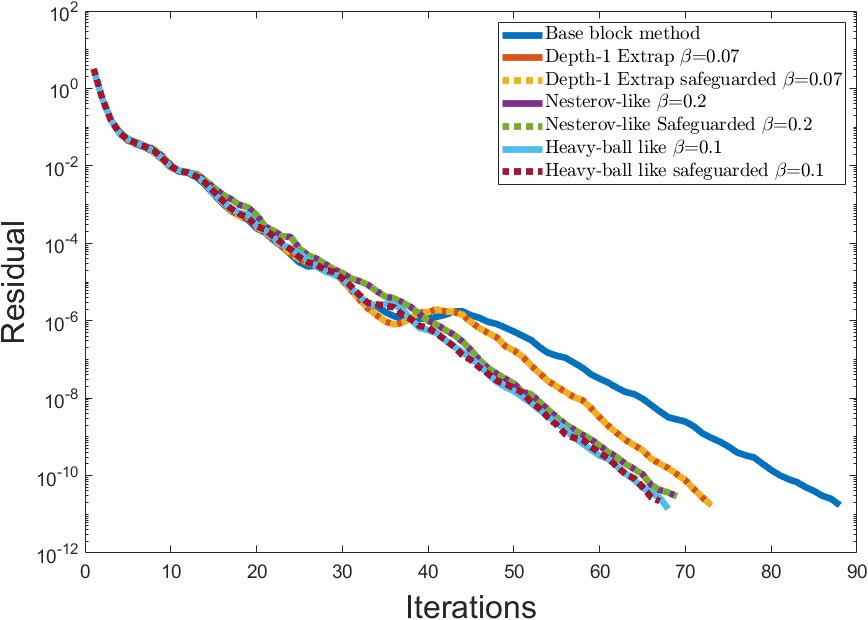}
        \includegraphics[width=0.49\textwidth]{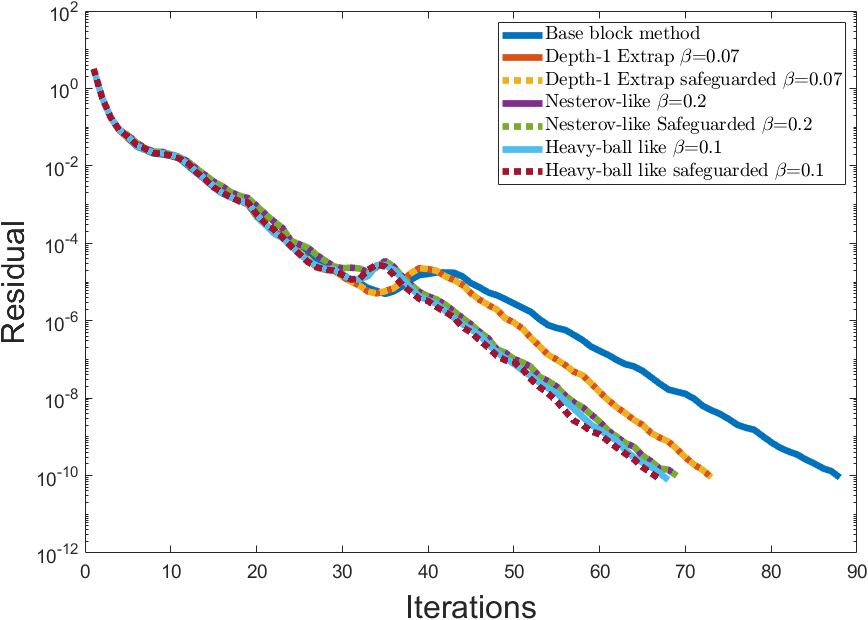}
        \includegraphics[width=0.49\textwidth]{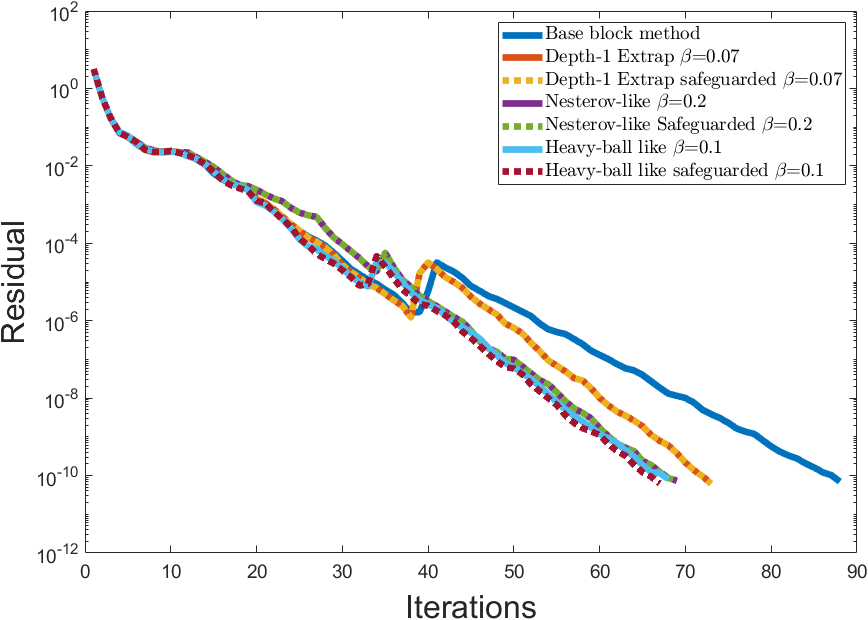}
        \includegraphics[width=0.49\textwidth]{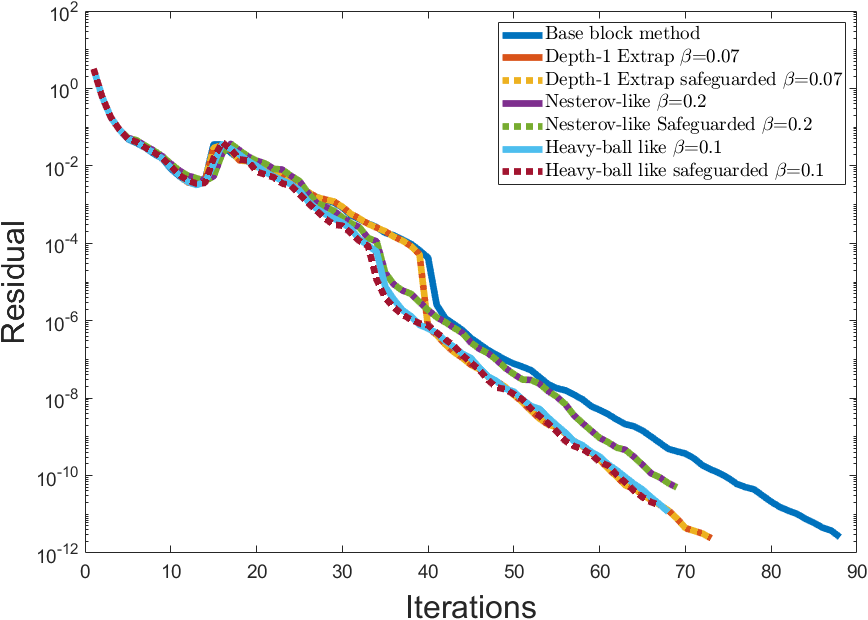}
        \caption{Residual convergence solving for the four smallest eigenvalues, $m=5$. Top left: $\lambda_1$, top right: $\lambda_2$, bottom left: $\lambda_3$, bottom right: $\lambda_4$.}
    \end{figure}

    \subsection{Further results}

    We present averaged results for seven eigenvalue problems, including the previous barbell and four-barbell examples. For the five new problems, we use $B=I$ with an $n\times n$ symmetric matrix from the SuiteSparse matrix collection \cite{sparsesuite} as $A$. The matrices used are {\tt SiH4} with $n=5041$, {\tt SiNA} with $n=5743$, the negation of {\tt benzene} with $n=8217$, {\tt Si10H16} with $n=17077$, and {\tt Gee99H100} with $n=112985$. Each eigenvalue problem is solved by the base and accelerated methods using the same randomly generated initial iterate for $b$ many algebraically smallest eigenvalues with a tolerance of $10^{-10}$. The blocksize $b$ is chosen for each problem to include clustered eigenvalues. For some matrices, the smallest eigenvalue is not in the cluster. For example, for matrix {\tt SiH4}, the distance between the first and second smallest eigenvalues is $0.37$ while the distance between the second and third smallest eigenvalues is $3.76e^{-12}$. For {\tt SiNa}, {\tt SiH4}, and the negation of {\tt benzene}, the distance between the smallest eigenvalue and $b^{th}$ smallest eigenvalue is relatively large with distances around $0.4$. {\tt Si10H16} and {\tt Ge99H100} had distances of $0.088$ and $0.028$ respectively. For the barbell problems, the distance is quite small at around $10^{-4}$. We present the mean and standard deviation of 50 trials for each problem in Table 1 for $m=1$, Table 2 for $m=2$, and Table 3 for $m=5$.
    
    We observed the following results. The Depth-1 and Nesterov-like methods performed similarly for each value of $\beta$ and performed best for $\beta=0.1$ or $\beta=0.25$. The heavy-ball-like method performed best for $\beta=0.1$, which usually converged in the fewest iterations of all the methods. For small $m$, we saw significant decrease in total iterations from accelerating the base method, but as $m$ increases, the accelerated methods perform more closely to the base method. The accelerated methods also saw more significant decreases in total iterations when the eigenvalues were more closely clustered and the size of the matrices were larger. The accelerated methods were also more stable with respect to initial iterate than the base method, which had higher standard deviations across all problems.

    \begin{table}[]
\centering
\caption{Average number of iterations (top) and standard deviation (bottom) to residual convergence of $10^{-10}$ with $m=1$ for various fixed $\beta$. Here, the base method is equivalent to LOBPCG.}
\begin{tabular}{|l|l|llll|llll|lll|}
\hline
Method & Base & \multicolumn{4}{l|}{Depth-1} & \multicolumn{4}{l|}{Nesterov-like} & \multicolumn{3}{l|}{Heavy-ball-like}\\ \hline
\backslashbox{\rule{0pt}{4ex}Problem, $b$\kern-1.8em}{\kern-1.8em$\beta$} & 0 & \multicolumn{1}{l|}{0.1} & \multicolumn{1}{l|}{0.25} & \multicolumn{1}{l|}{0.5} & 0.75 & \multicolumn{1}{l|}{0.1} & \multicolumn{1}{l|}{0.25} & \multicolumn{1}{l|}{0.5} & 0.75 & \multicolumn{1}{l|}{0.1} & \multicolumn{1}{l|}{0.2} & 0.3\\ \hline
2-barbell, 2 & \begin{tabular}[c]{@{}l@{}}373\\ 75\end{tabular} & \multicolumn{1}{l|}{\begin{tabular}[c]{@{}l@{}}306\\ 21\end{tabular}} & \multicolumn{1}{l|}{\begin{tabular}[c]{@{}l@{}}314\\ 15\end{tabular}} & \multicolumn{1}{l|}{\begin{tabular}[c]{@{}l@{}}340\\ 18\end{tabular}} & \begin{tabular}[c]{@{}l@{}}367\\ 16\end{tabular} & \multicolumn{1}{l|}{\begin{tabular}[c]{@{}l@{}}308\\ 27\end{tabular}} & \multicolumn{1}{l|}{\begin{tabular}[c]{@{}l@{}}312\\ 12\end{tabular}} & \multicolumn{1}{l|}{\begin{tabular}[c]{@{}l@{}}339\\ 16\end{tabular}} & \begin{tabular}[c]{@{}l@{}}369\\ 16\end{tabular} & \multicolumn{1}{l|}{\begin{tabular}[c]{@{}l@{}}290\\ 25\end{tabular}} & \multicolumn{1}{l|}{\begin{tabular}[c]{@{}l@{}}325\\ 10\end{tabular}} & \begin{tabular}[c]{@{}l@{}}362\\ 13\end{tabular} \\ \hline
{\tt SiH4}, 4                               & \begin{tabular}[c]{@{}l@{}}148\\ 18\end{tabular}   & \multicolumn{1}{l|}{\begin{tabular}[c]{@{}l@{}}120\\ 7\end{tabular}}   & \multicolumn{1}{l|}{\begin{tabular}[c]{@{}l@{}}123\\ 4\end{tabular}}   & \multicolumn{1}{l|}{\begin{tabular}[c]{@{}l@{}}133\\ 4\end{tabular}}   & \begin{tabular}[c]{@{}l@{}}140\\ 4\end{tabular}   & \multicolumn{1}{l|}{\begin{tabular}[c]{@{}l@{}}119\\ 5\end{tabular}}   & \multicolumn{1}{l|}{\begin{tabular}[c]{@{}l@{}}123\\ 4\end{tabular}}   & \multicolumn{1}{l|}{\begin{tabular}[c]{@{}l@{}}131\\ 4\end{tabular}}   & \begin{tabular}[c]{@{}l@{}}140\\ 4\end{tabular}   & \multicolumn{1}{l|}{\begin{tabular}[c]{@{}l@{}}113\\ 5\end{tabular}}   & \multicolumn{1}{l|}{\begin{tabular}[c]{@{}l@{}}120\\ 3\end{tabular}}   & \begin{tabular}[c]{@{}l@{}}64\\ 131\end{tabular}   \\ \hline
{\tt SiNa}, 3                               & \begin{tabular}[c]{@{}l@{}}273\\ 43\end{tabular} & \multicolumn{1}{l|}{\begin{tabular}[c]{@{}l@{}}249\\ 21\end{tabular}} & \multicolumn{1}{l|}{\begin{tabular}[c]{@{}l@{}}267\\ 19\end{tabular}}  & \multicolumn{1}{l|}{\begin{tabular}[c]{@{}l@{}}296\\ 20\end{tabular}}  & \begin{tabular}[c]{@{}l@{}}321\\ 21\end{tabular} & \multicolumn{1}{l|}{\begin{tabular}[c]{@{}l@{}}250\\ 24\end{tabular}}  & \multicolumn{1}{l|}{\begin{tabular}[c]{@{}l@{}}263\\ 20\end{tabular}} & \multicolumn{1}{l|}{\begin{tabular}[c]{@{}l@{}}294\\ 22\end{tabular}} & \begin{tabular}[c]{@{}l@{}}322\\ 21\end{tabular}  & \multicolumn{1}{l|}{\begin{tabular}[c]{@{}l@{}}240\\ 18\end{tabular}} & \multicolumn{1}{l|}{\begin{tabular}[c]{@{}l@{}}263\\ 18\end{tabular}}  & \begin{tabular}[c]{@{}l@{}}286\\ 18\end{tabular} \\ \hline
{\tt benzene}, 3                            & \begin{tabular}[c]{@{}l@{}}197\\ 21\end{tabular}   & \multicolumn{1}{l|}{\begin{tabular}[c]{@{}l@{}}174\\ 11\end{tabular}}   & \multicolumn{1}{l|}{\begin{tabular}[c]{@{}l@{}}177\\ 9\end{tabular}}   & \multicolumn{1}{l|}{\begin{tabular}[c]{@{}l@{}}191\\ 10\end{tabular}}   & \begin{tabular}[c]{@{}l@{}}203\\ 10\end{tabular}   & \multicolumn{1}{l|}{\begin{tabular}[c]{@{}l@{}}172\\ 10\end{tabular}}   & \multicolumn{1}{l|}{\begin{tabular}[c]{@{}l@{}}179\\ 10\end{tabular}}   & \multicolumn{1}{l|}{\begin{tabular}[c]{@{}l@{}}191\\ 10\end{tabular}}   & \begin{tabular}[c]{@{}l@{}}202\\ 9\end{tabular}   & \multicolumn{1}{l|}{\begin{tabular}[c]{@{}l@{}}164\\ 8\end{tabular}}   & \multicolumn{1}{l|}{\begin{tabular}[c]{@{}l@{}}179\\ 8\end{tabular}}   & \begin{tabular}[c]{@{}l@{}}198\\ 9\end{tabular}   \\ \hline
4-barbell, 4 & \begin{tabular}[c]{@{}l@{}}617\\ 78\end{tabular} & \multicolumn{1}{l|}{\begin{tabular}[c]{@{}l@{}}399\\ 47\end{tabular}} & \multicolumn{1}{l|}{\begin{tabular}[c]{@{}l@{}}382\\ 15\end{tabular}} & \multicolumn{1}{l|}{\begin{tabular}[c]{@{}l@{}}427\\ 16\end{tabular}} & \begin{tabular}[c]{@{}l@{}}472\\ 16\end{tabular} & \multicolumn{1}{l|}{\begin{tabular}[c]{@{}l@{}}411\\ 56\end{tabular}} & \multicolumn{1}{l|}{\begin{tabular}[c]{@{}l@{}}380\\ 15\end{tabular}} & \multicolumn{1}{l|}{\begin{tabular}[c]{@{}l@{}}428\\ 14\end{tabular}} & \begin{tabular}[c]{@{}l@{}}472\\ 17\end{tabular} & \multicolumn{1}{l|}{\begin{tabular}[c]{@{}l@{}}377\\ 41\end{tabular}} & \multicolumn{1}{l|}{\begin{tabular}[c]{@{}l@{}}389\\ 14\end{tabular}} & \begin{tabular}[c]{@{}l@{}}460\\ 17\end{tabular} \\ \hline
{\tt Si10H16}, 4                            & \begin{tabular}[c]{@{}l@{}}392\\ 45\end{tabular} & \multicolumn{1}{l|}{\begin{tabular}[c]{@{}l@{}}279\\ 18\end{tabular}} & \multicolumn{1}{l|}{\begin{tabular}[c]{@{}l@{}}285\\ 11\end{tabular}}  & \multicolumn{1}{l|}{\begin{tabular}[c]{@{}l@{}}121\\ 3\end{tabular}}  & \begin{tabular}[c]{@{}l@{}}348\\ 13\end{tabular}  & \multicolumn{1}{l|}{\begin{tabular}[c]{@{}l@{}}280\\ 22\end{tabular}}  & \multicolumn{1}{l|}{\begin{tabular}[c]{@{}l@{}}289\\ 13\end{tabular}}  & \multicolumn{1}{l|}{\begin{tabular}[c]{@{}l@{}}323\\ 12\end{tabular}}  & \begin{tabular}[c]{@{}l@{}}349\\ 13\end{tabular}  & \multicolumn{1}{l|}{\begin{tabular}[c]{@{}l@{}}265\\ 16\end{tabular}}  & \multicolumn{1}{l|}{\begin{tabular}[c]{@{}l@{}}293\\ 11\end{tabular}}  & \begin{tabular}[c]{@{}l@{}}336\\ 14\end{tabular}  \\ \hline
{\tt Ge99H100}, 4                            & \begin{tabular}[c]{@{}l@{}}754\\ 118\end{tabular} & \multicolumn{1}{l|}{\begin{tabular}[c]{@{}l@{}}595\\ 53\end{tabular}} & \multicolumn{1}{l|}{\begin{tabular}[c]{@{}l@{}}548\\ 32\end{tabular}} & \multicolumn{1}{l|}{\begin{tabular}[c]{@{}l@{}}607\\ 27\end{tabular}} & \begin{tabular}[c]{@{}l@{}}662\\ 23\end{tabular} & \multicolumn{1}{l|}{\begin{tabular}[c]{@{}l@{}}615\\ 67\end{tabular}} & \multicolumn{1}{l|}{\begin{tabular}[c]{@{}l@{}}549\\ 36\end{tabular}} & \multicolumn{1}{l|}{\begin{tabular}[c]{@{}l@{}}616\\ 28\end{tabular}}  & \begin{tabular}[c]{@{}l@{}}664\\ 21\end{tabular} & \multicolumn{1}{l|}{\begin{tabular}[c]{@{}l@{}}581\\ 47\end{tabular}} & \multicolumn{1}{l|}{\begin{tabular}[c]{@{}l@{}}564\\ 31\end{tabular}} & \begin{tabular}[c]{@{}l@{}}646\\ 24\end{tabular} \\ \hline
\end{tabular}
\end{table}

    For $m=1$, as the block size is fixed and we include the previous block of iterations, the base method is equivalent to LOBPCG of \cite{Knyazev_2000}, as noted in \cite{Quillen_Ye_2010}. The accelerated methods converged in fewer iterations than the base method on average for each problem for smaller $\beta$. For the smaller problems (barbell, {\tt SiNa}, {\tt SiH4}, and the negation of {\tt benzene}), the accelerated methods performed best on average for $\beta=0.1$, with the Depth-1 and Nesterov-like methods converging between $80\%-92\%$ of the base method's total iterations and the heavy-ball-like between $76\%-88\%$ of the base method's total iterations. For the four barbell problem, the Depth-1 and Nesterov-like methods performed best on average for $\beta=0.25$, converging in $62\%$ of the base method's total iterations, while the heavy-ball-like method performed best on average for $\beta=0.1$, converging in $61\%$ of the base method's total iterations. For {\tt Si10H16}, all three methods performed best on average with $\beta=.01$ in $68\%-71\%$ of the base method's total iterations. Lastly, for {\tt Ge99H100}, all three methods performed best on average with $\beta=.25$ in $73\%-75\%$ of the base method's total iterations.

    \begin{table}[]
\centering
\caption{Average number of iterations (top) and standard deviation (bottom) to residual convergence of $10^{-10}$ with $m=2$ for various fixed $\beta$.}
\begin{tabular}{|l|l|llll|llll|lll|}
\hline
Method & Base & \multicolumn{4}{l|}{Depth-1} & \multicolumn{4}{l|}{Nesterov-like} & \multicolumn{3}{l|}{Heavy-ball-like}\\ \hline
\backslashbox{\rule{0pt}{4ex}Problem, $b$\kern-1.8em}{\kern-1.8em$\beta$} & 0 & \multicolumn{1}{l|}{0.1} & \multicolumn{1}{l|}{0.25} & \multicolumn{1}{l|}{0.5} & 0.75 & \multicolumn{1}{l|}{0.1} & \multicolumn{1}{l|}{0.25} & \multicolumn{1}{l|}{0.5} & 0.75 & \multicolumn{1}{l|}{0.1} & \multicolumn{1}{l|}{0.2} & 0.3\\ \hline
2-barbell, 2                         & \begin{tabular}[c]{@{}l@{}}174\\ 36\end{tabular} & \multicolumn{1}{l|}{\begin{tabular}[c]{@{}l@{}}151\\ 10\end{tabular}} & \multicolumn{1}{l|}{\begin{tabular}[c]{@{}l@{}}154\\ 7\end{tabular}}  & \multicolumn{1}{l|}{\begin{tabular}[c]{@{}l@{}}166\\ 7\end{tabular}}  & \begin{tabular}[c]{@{}l@{}}173\\ 8\end{tabular}  & \multicolumn{1}{l|}{\begin{tabular}[c]{@{}l@{}}154\\ 10\end{tabular}} & \multicolumn{1}{l|}{\begin{tabular}[c]{@{}l@{}}153\\ 8\end{tabular}}  & \multicolumn{1}{l|}{\begin{tabular}[c]{@{}l@{}}164\\ 6\end{tabular}}  & \begin{tabular}[c]{@{}l@{}}178\\ 8\end{tabular}  & \multicolumn{1}{l|}{\begin{tabular}[c]{@{}l@{}}142\\ 9\end{tabular}}  & \multicolumn{1}{l|}{\begin{tabular}[c]{@{}l@{}}154\\ 5\end{tabular}}  & \begin{tabular}[c]{@{}l@{}}172\\ 5\end{tabular}  \\ \hline
{\tt SiH4}, 4                               & \begin{tabular}[c]{@{}l@{}}62\\ 5\end{tabular}   & \multicolumn{1}{l|}{\begin{tabular}[c]{@{}l@{}}59\\ 2\end{tabular}}   & \multicolumn{1}{l|}{\begin{tabular}[c]{@{}l@{}}66\\ 2\end{tabular}}   & \multicolumn{1}{l|}{\begin{tabular}[c]{@{}l@{}}66\\ 2\end{tabular}}   & \begin{tabular}[c]{@{}l@{}}70\\ 2\end{tabular}   & \multicolumn{1}{l|}{\begin{tabular}[c]{@{}l@{}}59\\ 3\end{tabular}}   & \multicolumn{1}{l|}{\begin{tabular}[c]{@{}l@{}}62\\ 2\end{tabular}}   & \multicolumn{1}{l|}{\begin{tabular}[c]{@{}l@{}}67\\ 2\end{tabular}}   & \begin{tabular}[c]{@{}l@{}}70\\ 2\end{tabular}   & \multicolumn{1}{l|}{\begin{tabular}[c]{@{}l@{}}57\\ 2\end{tabular}}   & \multicolumn{1}{l|}{\begin{tabular}[c]{@{}l@{}}59\\ 2\end{tabular}}   & \begin{tabular}[c]{@{}l@{}}64\\ 2\end{tabular}   \\ \hline
{\tt SiNa}, 3                               & \begin{tabular}[c]{@{}l@{}}119\\ 14\end{tabular} & \multicolumn{1}{l|}{\begin{tabular}[c]{@{}l@{}}119\\ 10\end{tabular}} & \multicolumn{1}{l|}{\begin{tabular}[c]{@{}l@{}}126\\ 9\end{tabular}}  & \multicolumn{1}{l|}{\begin{tabular}[c]{@{}l@{}}137\\ 9\end{tabular}}  & \begin{tabular}[c]{@{}l@{}}145\\ 10\end{tabular} & \multicolumn{1}{l|}{\begin{tabular}[c]{@{}l@{}}117\\ 9\end{tabular}}  & \multicolumn{1}{l|}{\begin{tabular}[c]{@{}l@{}}125\\ 10\end{tabular}} & \multicolumn{1}{l|}{\begin{tabular}[c]{@{}l@{}}135\\ 10\end{tabular}} & \begin{tabular}[c]{@{}l@{}}145\\ 9\end{tabular}  & \multicolumn{1}{l|}{\begin{tabular}[c]{@{}l@{}}115\\ 11\end{tabular}} & \multicolumn{1}{l|}{\begin{tabular}[c]{@{}l@{}}122\\ 8\end{tabular}}  & \begin{tabular}[c]{@{}l@{}}132\\ 10\end{tabular} \\ \hline
{\tt benzene}, 3                            & \begin{tabular}[c]{@{}l@{}}84\\ 8\end{tabular}   & \multicolumn{1}{l|}{\begin{tabular}[c]{@{}l@{}}82\\ 4\end{tabular}}   & \multicolumn{1}{l|}{\begin{tabular}[c]{@{}l@{}}89\\ 4\end{tabular}}   & \multicolumn{1}{l|}{\begin{tabular}[c]{@{}l@{}}93\\ 4\end{tabular}}   & \begin{tabular}[c]{@{}l@{}}99\\ 4\end{tabular}   & \multicolumn{1}{l|}{\begin{tabular}[c]{@{}l@{}}82\\ 4\end{tabular}}   & \multicolumn{1}{l|}{\begin{tabular}[c]{@{}l@{}}87\\ 4\end{tabular}}   & \multicolumn{1}{l|}{\begin{tabular}[c]{@{}l@{}}93\\ 5\end{tabular}}   & \begin{tabular}[c]{@{}l@{}}99\\ 4\end{tabular}   & \multicolumn{1}{l|}{\begin{tabular}[c]{@{}l@{}}80\\ 3\end{tabular}}   & \multicolumn{1}{l|}{\begin{tabular}[c]{@{}l@{}}85\\ 3\end{tabular}}   & \begin{tabular}[c]{@{}l@{}}92\\ 3\end{tabular}   \\ \hline
4-barbell, 4                          & \begin{tabular}[c]{@{}l@{}}269\\ 25\end{tabular} & \multicolumn{1}{l|}{\begin{tabular}[c]{@{}l@{}}195\\ 16\end{tabular}} & \multicolumn{1}{l|}{\begin{tabular}[c]{@{}l@{}}187\\ 9\end{tabular}}  & \multicolumn{1}{l|}{\begin{tabular}[c]{@{}l@{}}201\\ 11\end{tabular}} & \begin{tabular}[c]{@{}l@{}}216\\ 6\end{tabular}  & \multicolumn{1}{l|}{\begin{tabular}[c]{@{}l@{}}198\\ 21\end{tabular}} & \multicolumn{1}{l|}{\begin{tabular}[c]{@{}l@{}}188\\ 6\end{tabular}}  & \multicolumn{1}{l|}{\begin{tabular}[c]{@{}l@{}}200\\ 6\end{tabular}}  & \begin{tabular}[c]{@{}l@{}}216\\ 6\end{tabular}  & \multicolumn{1}{l|}{\begin{tabular}[c]{@{}l@{}}179\\ 13\end{tabular}} & \multicolumn{1}{l|}{\begin{tabular}[c]{@{}l@{}}183\\ 6\end{tabular}}  & \begin{tabular}[c]{@{}l@{}}208\\ 7\end{tabular}  \\ \hline
{\tt Si10H16}, 4                            & \begin{tabular}[c]{@{}l@{}}166\\ 23\end{tabular} & \multicolumn{1}{l|}{\begin{tabular}[c]{@{}l@{}}135\\ 10\end{tabular}} & \multicolumn{1}{l|}{\begin{tabular}[c]{@{}l@{}}138\\ 6\end{tabular}}  & \multicolumn{1}{l|}{\begin{tabular}[c]{@{}l@{}}152\\ 6\end{tabular}}  & \begin{tabular}[c]{@{}l@{}}162\\ 6\end{tabular}  & \multicolumn{1}{l|}{\begin{tabular}[c]{@{}l@{}}133\\ 7\end{tabular}}  & \multicolumn{1}{l|}{\begin{tabular}[c]{@{}l@{}}139\\ 7\end{tabular}}  & \multicolumn{1}{l|}{\begin{tabular}[c]{@{}l@{}}152\\ 5\end{tabular}}  & \begin{tabular}[c]{@{}l@{}}162\\ 5\end{tabular}  & \multicolumn{1}{l|}{\begin{tabular}[c]{@{}l@{}}128\\ 4\end{tabular}}  & \multicolumn{1}{l|}{\begin{tabular}[c]{@{}l@{}}139\\ 5\end{tabular}}  & \begin{tabular}[c]{@{}l@{}}154\\ 5\end{tabular}  \\ \hline
{\tt Ge99H100}, 4                           & \begin{tabular}[c]{@{}l@{}}329\\ 28\end{tabular} & \multicolumn{1}{l|}{\begin{tabular}[c]{@{}l@{}}279\\ 23\end{tabular}} & \multicolumn{1}{l|}{\begin{tabular}[c]{@{}l@{}}257\\ 12\end{tabular}} & \multicolumn{1}{l|}{\begin{tabular}[c]{@{}l@{}}278\\ 10\end{tabular}} & \begin{tabular}[c]{@{}l@{}}297\\ 12\end{tabular} & \multicolumn{1}{l|}{\begin{tabular}[c]{@{}l@{}}276\\ 21\end{tabular}} & \multicolumn{1}{l|}{\begin{tabular}[c]{@{}l@{}}256\\ 10\end{tabular}} & \multicolumn{1}{l|}{\begin{tabular}[c]{@{}l@{}}276\\ 9\end{tabular}}  & \begin{tabular}[c]{@{}l@{}}298\\ 10\end{tabular} & \multicolumn{1}{l|}{\begin{tabular}[c]{@{}l@{}}260\\ 16\end{tabular}} & \multicolumn{1}{l|}{\begin{tabular}[c]{@{}l@{}}255\\ 10\end{tabular}} & \begin{tabular}[c]{@{}l@{}}290\\ 11\end{tabular} \\ \hline
\end{tabular}
\end{table}
    
    For $m=2$, the acceleration methods performed similarly or worse to the base method for {\tt SiNa}, {\tt SiH4}, and the negation of {\tt benzene}. For {\tt Si10H16}, {\tt Ge99H100}, and the barbell problems, the accelerated methods converged on average in fewer iterations for all test values of $\beta$. For {\tt Si10H16}, {\tt Ge99H100}, and the barbell problems, the Depth-1 and Nesterov-like methods performed best on average for $\beta=0.1$ or $\beta=0.25$, converging between $72\%-88\%$ of the base method's total iterations, while the heavy-ball-like method performed best on average for $\beta=0.1$, converging between $67\%-81\%$ of the base method's total iterations.

    \begin{table}[]
\centering
\caption{Average number of iterations (top) and standard deviation (bottom) to residual convergence of $10^{-10}$ with $m=5$ for various fixed $\beta$.}
\begin{tabular}{|l|l|llll|llll|lll|}
\hline
Method & Base & \multicolumn{4}{l|}{Depth-1} & \multicolumn{4}{l|}{Nesterov-like} & \multicolumn{3}{l|}{Heavy-ball-like}\\ \hline
\backslashbox{\rule{0pt}{4ex}Problem, $b$\kern-1.8em}{\kern-1.8em$\beta$} & 0 & \multicolumn{1}{l|}{0.1} & \multicolumn{1}{l|}{0.25} & \multicolumn{1}{l|}{0.5} & 0.75 & \multicolumn{1}{l|}{0.1} & \multicolumn{1}{l|}{0.25} & \multicolumn{1}{l|}{0.5} & 0.75 & \multicolumn{1}{l|}{0.1} & \multicolumn{1}{l|}{0.2} & 0.3\\ \hline
2-barbell, 2                         & \begin{tabular}[c]{@{}l@{}}57\\ 4\end{tabular} & \multicolumn{1}{l|}{\begin{tabular}[c]{@{}l@{}}57\\ 2\end{tabular}} & \multicolumn{1}{l|}{\begin{tabular}[c]{@{}l@{}}61\\ 2\end{tabular}} & \multicolumn{1}{l|}{\begin{tabular}[c]{@{}l@{}}66\\ 2\end{tabular}} & \begin{tabular}[c]{@{}l@{}}69\\ 2\end{tabular} & \multicolumn{1}{l|}{\begin{tabular}[c]{@{}l@{}}58\\ 3\end{tabular}} & \multicolumn{1}{l|}{\begin{tabular}[c]{@{}l@{}}61\\ 3\end{tabular}} & \multicolumn{1}{l|}{\begin{tabular}[c]{@{}l@{}}66\\ 3\end{tabular}} & \begin{tabular}[c]{@{}l@{}}70\\ 3\end{tabular} & \multicolumn{1}{l|}{\begin{tabular}[c]{@{}l@{}}54\\ 2\end{tabular}} & \multicolumn{1}{l|}{\begin{tabular}[c]{@{}l@{}}59\\ 2\end{tabular}} & \begin{tabular}[c]{@{}l@{}}63\\ 2\end{tabular} \\ \hline
{\tt SiH4}, 4                               & \begin{tabular}[c]{@{}l@{}}23\\ 1\end{tabular} & \multicolumn{1}{l|}{\begin{tabular}[c]{@{}l@{}}24\\ 1\end{tabular}} & \multicolumn{1}{l|}{\begin{tabular}[c]{@{}l@{}}26\\ 1\end{tabular}} & \multicolumn{1}{l|}{\begin{tabular}[c]{@{}l@{}}29\\ 1\end{tabular}} & \begin{tabular}[c]{@{}l@{}}31\\ 1\end{tabular} & \multicolumn{1}{l|}{\begin{tabular}[c]{@{}l@{}}24\\ 1\end{tabular}} & \multicolumn{1}{l|}{\begin{tabular}[c]{@{}l@{}}26\\ 1\end{tabular}} & \multicolumn{1}{l|}{\begin{tabular}[c]{@{}l@{}}29\\ 1\end{tabular}} & \begin{tabular}[c]{@{}l@{}}31\\ 1\end{tabular} & \multicolumn{1}{l|}{\begin{tabular}[c]{@{}l@{}}23\\ 1\end{tabular}} & \multicolumn{1}{l|}{\begin{tabular}[c]{@{}l@{}}25\\ 1\end{tabular}} & \begin{tabular}[c]{@{}l@{}}28\\ 1\end{tabular} \\ \hline
{\tt SiNa}, 3                               & \begin{tabular}[c]{@{}l@{}}43\\ 4\end{tabular} & \multicolumn{1}{l|}{\begin{tabular}[c]{@{}l@{}}45\\ 3\end{tabular}} & \multicolumn{1}{l|}{\begin{tabular}[c]{@{}l@{}}48\\ 3\end{tabular}} & \multicolumn{1}{l|}{\begin{tabular}[c]{@{}l@{}}50\\ 3\end{tabular}} & \begin{tabular}[c]{@{}l@{}}53\\ 3\end{tabular} & \multicolumn{1}{l|}{\begin{tabular}[c]{@{}l@{}}45\\ 4\end{tabular}} & \multicolumn{1}{l|}{\begin{tabular}[c]{@{}l@{}}48\\ 3\end{tabular}} & \multicolumn{1}{l|}{\begin{tabular}[c]{@{}l@{}}50\\ 3\end{tabular}} & \begin{tabular}[c]{@{}l@{}}53\\ 3\end{tabular} & \multicolumn{1}{l|}{\begin{tabular}[c]{@{}l@{}}46\\ 3\end{tabular}} & \multicolumn{1}{l|}{\begin{tabular}[c]{@{}l@{}}46\\ 3\end{tabular}} & \begin{tabular}[c]{@{}l@{}}49\\ 3\end{tabular} \\ \hline
{\tt benzene}, 3                            & \begin{tabular}[c]{@{}l@{}}31\\ 2\end{tabular} & \multicolumn{1}{l|}{\begin{tabular}[c]{@{}l@{}}33\\ 1\end{tabular}} & \multicolumn{1}{l|}{\begin{tabular}[c]{@{}l@{}}36\\ 2\end{tabular}} & \multicolumn{1}{l|}{\begin{tabular}[c]{@{}l@{}}39\\ 1\end{tabular}} & \begin{tabular}[c]{@{}l@{}}41\\ 2\end{tabular} & \multicolumn{1}{l|}{\begin{tabular}[c]{@{}l@{}}33\\ 1\end{tabular}} & \multicolumn{1}{l|}{\begin{tabular}[c]{@{}l@{}}36\\ 1\end{tabular}} & \multicolumn{1}{l|}{\begin{tabular}[c]{@{}l@{}}39\\ 1\end{tabular}} & \begin{tabular}[c]{@{}l@{}}42\\ 2\end{tabular} & \multicolumn{1}{l|}{\begin{tabular}[c]{@{}l@{}}32\\ 1\end{tabular}} & \multicolumn{1}{l|}{\begin{tabular}[c]{@{}l@{}}34\\ 1\end{tabular}} & \begin{tabular}[c]{@{}l@{}}37\\ 1\end{tabular} \\ \hline
4-barbell, 4                          & \begin{tabular}[c]{@{}l@{}}81\\ 8\end{tabular} & \multicolumn{1}{l|}{\begin{tabular}[c]{@{}l@{}}71\\ 4\end{tabular}} & \multicolumn{1}{l|}{\begin{tabular}[c]{@{}l@{}}73\\ 3\end{tabular}} & \multicolumn{1}{l|}{\begin{tabular}[c]{@{}l@{}}78\\ 3\end{tabular}} & \begin{tabular}[c]{@{}l@{}}83\\ 2\end{tabular} & \multicolumn{1}{l|}{\begin{tabular}[c]{@{}l@{}}70\\ 3\end{tabular}} & \multicolumn{1}{l|}{\begin{tabular}[c]{@{}l@{}}72\\ 3\end{tabular}} & \multicolumn{1}{l|}{\begin{tabular}[c]{@{}l@{}}79\\ 3\end{tabular}} & \begin{tabular}[c]{@{}l@{}}83\\ 2\end{tabular} & \multicolumn{1}{l|}{\begin{tabular}[c]{@{}l@{}}67\\ 3\end{tabular}} & \multicolumn{1}{l|}{\begin{tabular}[c]{@{}l@{}}70\\ 2\end{tabular}} & \begin{tabular}[c]{@{}l@{}}76\\ 3\end{tabular} \\ \hline
{\tt Si10H16}, 4                            & \begin{tabular}[c]{@{}l@{}}53\\ 5\end{tabular} & \multicolumn{1}{l|}{\begin{tabular}[c]{@{}l@{}}51\\ 2\end{tabular}} & \multicolumn{1}{l|}{\begin{tabular}[c]{@{}l@{}}54\\ 2\end{tabular}} & \multicolumn{1}{l|}{\begin{tabular}[c]{@{}l@{}}59\\ 2\end{tabular}} & \begin{tabular}[c]{@{}l@{}}63\\ 2\end{tabular} & \multicolumn{1}{l|}{\begin{tabular}[c]{@{}l@{}}51\\ 2\end{tabular}} & \multicolumn{1}{l|}{\begin{tabular}[c]{@{}l@{}}54\\ 2\end{tabular}} & \multicolumn{1}{l|}{\begin{tabular}[c]{@{}l@{}}59\\ 2\end{tabular}} & \begin{tabular}[c]{@{}l@{}}63\\ 2\end{tabular} & \multicolumn{1}{l|}{\begin{tabular}[c]{@{}l@{}}51\\ 2\end{tabular}} & \multicolumn{1}{l|}{\begin{tabular}[c]{@{}l@{}}54\\ 2\end{tabular}} & \begin{tabular}[c]{@{}l@{}}58\\ 2\end{tabular} \\ \hline
{\tt Ge99H100}, 4                           & \begin{tabular}[c]{@{}l@{}}111\\ 12\end{tabular} & \multicolumn{1}{l|}{\begin{tabular}[c]{@{}l@{}}98\\ 6\end{tabular}} & \multicolumn{1}{l|}{\begin{tabular}[c]{@{}l@{}}99\\ 3\end{tabular}} & \multicolumn{1}{l|}{\begin{tabular}[c]{@{}l@{}}106\\ 4\end{tabular}} & \begin{tabular}[c]{@{}l@{}}114\\ 4\end{tabular} & \multicolumn{1}{l|}{\begin{tabular}[c]{@{}l@{}}96\\ 6\end{tabular}} & \multicolumn{1}{l|}{\begin{tabular}[c]{@{}l@{}}99\\ 5\end{tabular}} & \multicolumn{1}{l|}{\begin{tabular}[c]{@{}l@{}}106\\ 4\end{tabular}} & \begin{tabular}[c]{@{}l@{}}113\\ 4\end{tabular} & \multicolumn{1}{l|}{\begin{tabular}[c]{@{}l@{}}91\\ 5\end{tabular}} & \multicolumn{1}{l|}{\begin{tabular}[c]{@{}l@{}}95\\ 3\end{tabular}} & \begin{tabular}[c]{@{}l@{}}104\\ 3\end{tabular} \\ \hline

\end{tabular}
\end{table}

    For $m=5$, we observe less advantage from using the accelerated methods. For {\tt Si10H16} and the smaller dimension problems, the accelerated methods converge similarly or worse than the base method. However, the accelerated methods still perform well for the four-barbell problem and for {\tt Ge99H100}. For both problems, the Depth-1 and Nesterov-like methods performed best on average with $\beta=0.1$ and converged in around $87\%$ of the base method's total iterations, while the heavy-ball-like method performed best on average with $\beta=0.1$ and converged in around $82\%$ of the base method's total iterations.

    \subsection{Example 3}

    For our final example, we use the negation of the $61349\times 61349$ matrix {\tt GaAsH6} from the SparseSuite collection with $B=I$ and solve for the two smallest eigenvalues with a tolerance of $10^{-10}$. The distance between the smallest eigenvalue and second smallest eigenvalue is significantly small at $3.18e^{-12}$. We use the same random initial iterate for each method and use $\beta=0.1$ seen for the fixed methods and $\beta_0=0.1$ for the safeguarded methods. We declare convergence when the residuals for both approximate eigenpairs are less than Results are shown in Figure 12.

    \begin{figure}
        \centering
        \includegraphics[width=0.49\textwidth]{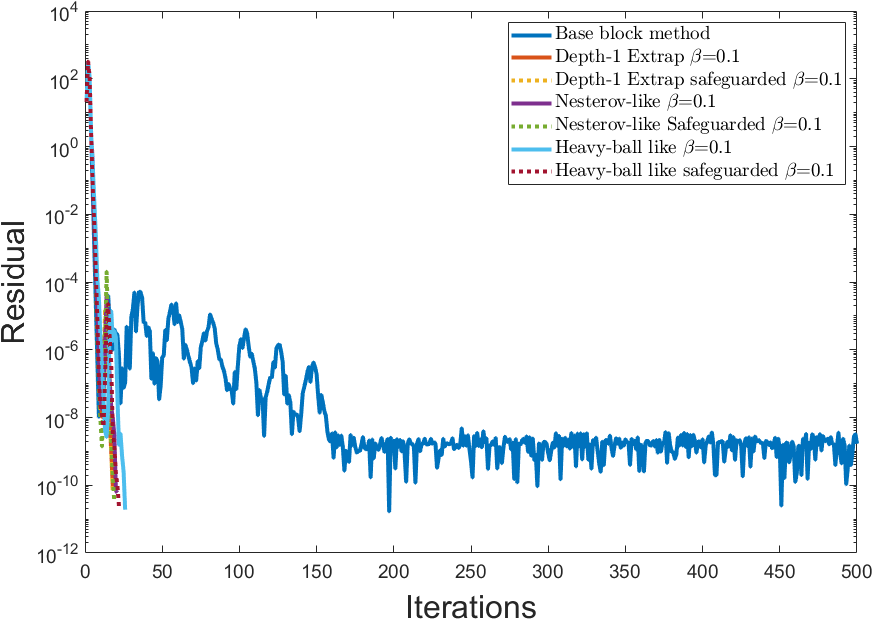}
        \includegraphics[width=0.49\textwidth]{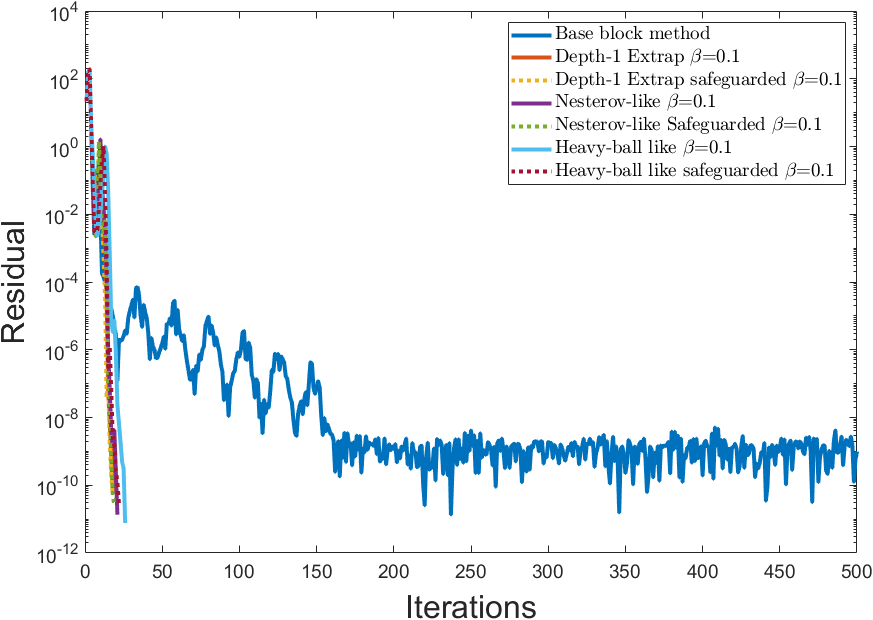}
        \includegraphics[width=0.49\textwidth]{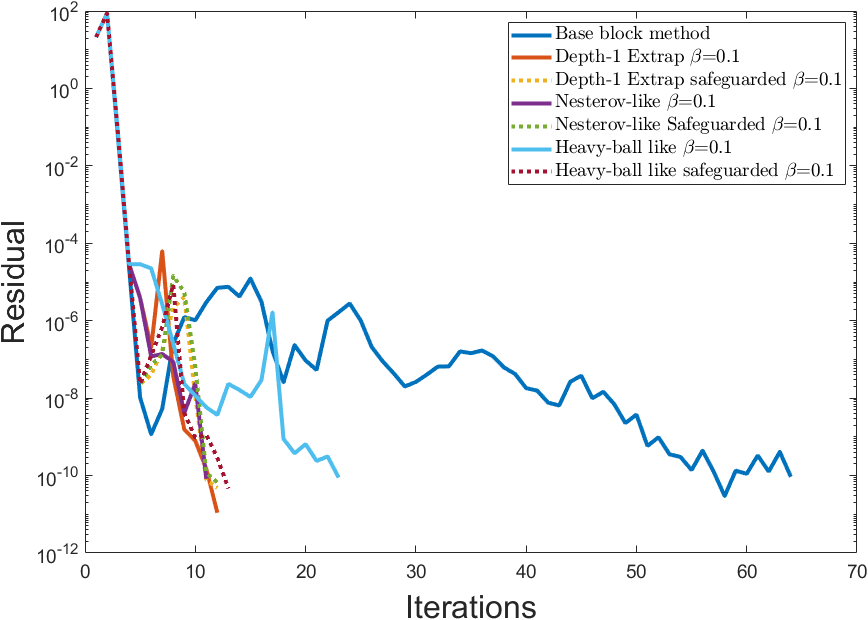}
        \includegraphics[width=0.49\textwidth]{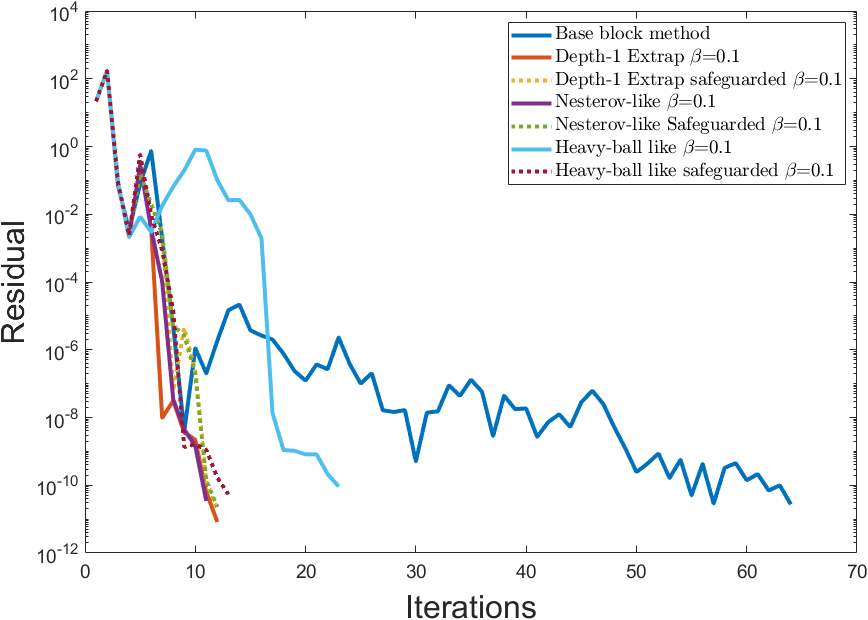}
        \caption{Residual convergence solving for the two smallest eigenvalues. Top left: $m=1$, $\lambda_1$, top right: $m=1$, $\lambda_2$, bottom left: $m=2$, $\lambda_1$, bottom right: $m=2$, $\lambda_2$. For $m=1$, the base method is equivalent to LOBPCG.}
    \end{figure}

    We observed the following. For $m=1$, where the base method is equivalent to LOBPCG, the accelerated methods performed better on average than the base method. For 50 trials, the base method converged in an average of 80 iterations with a standard deviation of 71. In 10 of the 50 trials, the base method failed to converge due to the noise regime, as shown in Figure 12. The Depth-1 and Nesterov-like methods again performed similarly to each other, converging in 42 and 30 iterations on average respectively with standard deviations 56 and 44 respectively. In five of the 50 trials, the fixed Depth-1 method failed to converge due to the noise regime, and in three of the 50 trials, the fixed Nesterov-like method failed to converge due to the noise regime, with the safeguarded methods behaving similarly. The fixed heavy-ball method usually performed slightly worse than the other fixed methods, but was much more stable with an average total of 26 iterations with a standard deviation of 3. The safeguarded heavy-ball method performed best with an average of 21 iterations and standard deviation of 2. The heavy-ball method never failed to converge for all observed trials.

    For $m=2$, we observed similar results with the base method converging on average in 76 iterations with a standard deviation of 50, and failing to converge once. The accelerated methods never failed to converge for all observed trials. The fixed Depth-1 and Nesterov-like methods both converged on average in 10 iterations with a standard deviation of 1, or $13\%$ of the base methods total iterations, with their respective adaptive safeguarded methods performing similarly. The fixed heavy-ball method converged on average in 21 iterations with standard deviation of 2, or $28\%$ of the base methods total iterations, while the safeguard adaptive heavy-ball method performed more closely to the other accelerated methods, converging on average in $17\%$ of the base methods total iterations. This can be seen in the bottom row of Figure 12.
    
\section{Conclusion}

    We provided an explanation for the acceleration seen when implementing Polyak's heavy-ball method
    on eigensolvers based on Krylov subspace projection methods. We introduced three acceleration schemes for the inverse-free Krylov subspace method and its block version based on the Nesterov and heavy-ball methods. A convergence theory has been shown further generalizing the choice of subspace for the generic eigensolver scheme described by Quillen and Ye, including for the three methods introduced in this paper, and analysis of extrapolation has been shown to justify the expected decrease in total iterations from the Krylov subspace method for the three schemes. In practice, the accelerated schemes show faster convergence when using small subspace sizes for a small cost in computation and storage. Convergence is significantly faster when solving for multiple heavily clustered eigenvalues. In the future, we plan to implement the deflation of found eigenvectors, a feature available in the black box implementation of the block inverse-free Krylov subspace method provided by Ye. We also plan to study how preconditioning affects the accelerated methods. 

\section*{Acknowledgments}

    Michelle Baker and Sara Pollock are supported in part by NSF DMS-2045059 (CAREER). Michelle Baker would like to thank Qiang Ye for his insight and guidance.

\bibliographystyle{plain}
\bibliography{bib}

\end{document}